\documentclass[amstex,a4paper,12pt,reqno]{amsart}
\usepackage{amssymb,amsmath,amsthm}
\usepackage[pdftex]{graphicx}

\usepackage{bm}
\usepackage{verbatim}

\def\R{{\mathbb R}}

\def\a{\alpha}
\def\b{\beta}
\def\d{\delta}

\def\g{\gamma}
\def\s{\sigma}
\def\t{\theta}
\def\l{\lambda}
\def\p{\partial}
\def\O{\Omega}
\def\o{\omega}
\def\e{\varepsilon}
\def\v{\varphi}
\def\G{\Gamma}

\def\mc{\mathcal}
\def\mf{\mathfrak}

\def\L{\Lambda}

\numberwithin{equation}{section}

\theoremstyle{definition}

\theoremstyle{plain}
\newtheorem{theorem}{Theorem}[section]
\newtheorem{proposition}{Proposition}[section]

\newtheorem{corollary}{Corollary}[section]

\theoremstyle{definition}

\begin{document}

\title[Principal eigenvalues for cooperative periodic-parabolic systems]
{Principal eigenvalue and maximum principle  for cooperative periodic-parabolic systems}

\author{I. Ant\'on}
\address{Departamento de An\'alisis Matem\'atico y Matem\'atica Aplicada,
Universidad Complutense de Madrid,
Madrid 28040, Spain}
\email{iantonlo@ucm.es}

\author{J. L\'opez-G\'omez}
\address{Instituto de Matem\'atica Interdisciplinar (IMI),
Departamento de An\'alisis Matem\'atico y Matem\'atica Aplicada,
Universidad Complutense de Madrid,
Madrid 28040, Spain}
\email{Lopez\_Gomez@mat.ucm.es}

\maketitle

\begin{abstract}
This paper classifies  the set of supersolutions of a general class of periodic-parabolic problems in the presence of a positive supersolution. From this result
we characterize the positivity of the underlying resolvent operator through the positivity of the associated principal eigenvalue and the existence of a positive strict supersolution. Lastly, this (scalar) characterization is used to characterize the strong maximum principle for a class of periodic-parabolic systems of cooperative type under arbitrary boundary conditions of mixed type.

\vspace{0.1cm}

\vspace{0.1cm}

\vspace{0.1cm}

\noindent This paper has been supported by the IMI of Complutense University and the Ministry of Economy and
Competitiveness of Spain under Grant MTM2015-65899-P
\end{abstract}

\vspace{0.2cm}

\section{Introduction}

\noindent This paper gives a periodic-parabolic counterpart of the second classification theorem of
J. L\'opez-G\'omez \cite{LG03} and infers from it a periodic-parabolic counterpart
of \cite[Th. 2.5]{LG96}  and  Theorem 2.4 of H. Amann and J. L\'opez-G\'omez \cite{AL98}. Then, based on that result,
the main theorem  of \cite{ALG}, which was originally stated for cooperative systems subject to   Dirichlet boundary conditions,  is substantially sharpened up to cover the case of general boundary operators of mixed type. The elliptic counterparts of these results have shown to be a milestone for the generation of new results in spatially heterogeneous nonlinear elliptic equations and cooperative systems (see, e.g.,  P. \'Alvarez-Caudevilla and
J. L\'opez-G\'omez \cite{ALG1,ALG2}, M. Molina-Meyer \cite{MM1,MM2,MM3}, H. Amann \cite{Am10} and the recent monograph \cite{LG15}). Thus, the findings of this paper seem imperative
for analyzing a wide variety of problems in the presence of spatio temporal heterogeneities.
\par
In this paper we are working under the following general assumptions:
\begin{enumerate}
\item[(A1)] $\O$ is a bounded subdomain (open and connected set) of $\R^N$, $N\geq 1$, of class $\mathcal{C}^{2+\t}$ for some $0<\t\leq 1$, whose boundary, $\p\O$, consists of two disjoint open and closed subsets, $\G_0$ and $\G_1$, respectively, such that  $\p\O :=\G_0\cup \G_1$ (as they are disjoint, $\G_0$ and $\G_1$ must be of class  $\mc{C}^{2+\t}$).

\item[(A2)] For a given $T>0$, we consider the non-autonomous differential operator
\begin{equation}
\label{1.1}
  \mathfrak{L}:=\mathfrak{L}(x,t):= -\sum_{i,j=1}^N a_{ij}(x,t)\frac{\p^2}
  {\p x_i \p x_j}+\sum_{j=1}^N b_j(x,t) \frac{\p}{\p x_j}+c(x,t)
\end{equation}
with $a_{ij}=a_{ji}, b_j, c \in F$ for all $1\leq i, j\leq N$, where
\begin{equation}
\label{1.2}
  F:= \left\{u\in \mathcal{C}^{\t,\frac{\t}{2}}(\bar\O\times \R;\R)\;: \;\;
  u(\cdot,T+t)=u(\cdot,t) \;\; \hbox{for all}\; t \in\R \right\}.
\end{equation}
Moreover, we assume that $\mathfrak{L}$ is uniformly elliptic in $\bar Q_T$, where $Q_T$ stands for the parabolic cylinder
\begin{equation*}
  Q_T:=\O \times (0,T),
\end{equation*}
i.e., there exists $\mu>0$ such that
\begin{equation*}
  \sum_{i,j=1}^N a_{ij}(x,t) \xi_i\xi_j\geq \mu\, |\xi|^2\quad \hbox{for all} \;\; (x,t,\xi)\in \bar Q_T\times \R^N,
\end{equation*}
where $|\cdot|$ stands for the  Euclidean norm of $\R^N$.

\item[(A3)] $\mathfrak{B}: \mathcal{C}(\G_0)\oplus\mathcal{C}^1(\O\cup \G_1) \to
C(\p\O)$ stands for the boundary operator
\begin{equation}
\label{1.3}
    \mathfrak{B} \xi := \left\{ \begin{array}{ll}
    \xi \qquad & \hbox{on } \;\;\G_0 \\     \frac{\p \xi}{\p \nu} + \b(x) \xi \qquad &
   \hbox{on } \;\; \G_1  \end{array} \right.
\end{equation}
for each $\xi\in \mathcal{C}(\G_0)\oplus \mathcal{C}^1(\O\cup\G_1)$, where $\b\in \mathcal{C}^{1+\t}(\G_1)$ and
\[
  \nu =(\nu_1,...,\nu_N)\in \mathcal{C}^{1+\t}(\p\O;\R^N)
\]
is an outward pointing nowhere tangent vector field.
\end{enumerate}

Thus, rather crucially, in this paper the function $\b$ can change sign, in strong contrast with the classical setting dealt with by P. Hess \cite{Hess} and, more recently, by R. Peng and X. Q. Zhao \cite{PZ}, where it was imposed the strongest condition $\b \geq 0$. In our general setting, $\mathfrak{B}$ is the \emph{Dirichlet boundary operator} on $\G_0$, and the \emph{Neumann}, or a \emph{first order regular oblique derivative boundary operator}, on $\G_1$, and either $\G_0$, or $\G_1$, can be empty. As in this paper $\b$ can change of sign, our results can be applied straight away to deal with cooperative periodic-parabolic systems under general \emph{nonlinear mixed boundary conditions} by the first time in the literature, which was a challenge. 
\par
Besides the space $F$ introduced in \eqref{1.2}, this paper also considers the Banach spaces of H\"{o}lder continuous $T$--periodic functions
\begin{equation*}
  E:= \left\{u\in \mathcal{C}^{2+\t,1+\frac{\t}{2}}(\bar\O\times \R;\R)\;: \quad  u(\cdot,T+t)=u(\cdot,t) \;\; \hbox{for all}\;\; t \in\R \right\}
\end{equation*}
and the periodic-parabolic operator
\begin{equation}
\label{i.4}
  \mathcal{P}:= \p_t + \mathfrak{L}(x,t).
\end{equation}
A function $h \in E$ is said to be a \emph{supersolution} of $(\mathcal{P},\mathfrak{B},Q_T)$ if
\begin{equation}
\label{i.5}
  \left\{ \begin{array}{ll} \mathcal{P}h \geq 0 &\quad
  \hbox{in}\;\;Q_T=\O\times (0,T), \\  \mathfrak{B}h \geq 0 &\quad\hbox{on}\;\;\p\O\times [0,T],\end{array}\right.
\end{equation}
and it is said to be a \emph{strict supersolution} of $(\mathcal{P},\mathfrak{B},Q_T)$ when, in addition, some of these inequalities is strict.
\par
The first goal of this paper is establishing the next periodic-parabolic counterpart of \cite[Th. 5.2]{LG03} (see also
Theorem 2.4 of \cite{LG13} and Theorem 2 of W. Walter \cite{Wa89}).

\begin{theorem}
\label{th11} Suppose $(\mathcal{P},\mathfrak{B},Q_T)$ admits a non-negative
supersolution $h\in E\setminus\{0\}$. Then, any supersolution $u\in E$ of $(\mathcal{P},\mathfrak{B},Q_T)$ satisfies one of the following alternatives:
\begin{enumerate}
\item[{\rm (a)}]  $u=0$ in $Q_T$.

\item[{\rm (b)}] $u(\cdot,t)\gg 0$ for all $t\in [0,T]$, in the sense that, for every $t\in [0,T]$, $u(x,t)>0$ for all $x\in\O\cup\G_1$ and
\begin{equation}
\label{i.6}
  \frac{\p u}{\p \nu}(x,t)<0\qquad \hbox{for every }\;\; x \in u^{-1}(0)\cap \G_0.
\end{equation}

\item[{\rm (c)}] There exists a constant $m<0$ such that $u = m h$ in $Q_T$. In such case,
\begin{equation}
\label{1.7}
  \left\{ \begin{array}{ll} \mathcal{P}h =0 &\quad
  \hbox{in}\;\;Q_T, \\  \mathfrak{B}h = 0  &\quad\hbox{on}\;\;\p\O\times [0,T].\end{array}\right.
\end{equation}
\end{enumerate}
\end{theorem}

In particular, $h$ satisfies Alternative (b) if it is a strict supersolution of $(\mathcal{P},\mathfrak{B},Q_T)$. Hence, if $(\mathcal{P},\mathfrak{B},Q_T)$ admits  a non-negative strict supersolution $h \in E\setminus\{0\}$, and $f\in E$ satisfies $f>0$, in the sense that $f\geq 0$ but $f\neq 0$, then, any ($T$-periodic) solution $u\in E$ of
\begin{equation}
\label{i.7}
  \left\{ \begin{array}{ll} \mathcal{P}u =f &\quad   \hbox{in}\;\;Q_T, \\
  \mathfrak{B}u = 0 &\quad\hbox{on}\;\;\p\O\times [0,T], \end{array}\right.
\end{equation}
satisfies $u(\cdot,t)\gg 0$ for all $t\in [0,T]$.  Consequently, the resolvent operator  $\mathcal{P}^{-1}:F\to F$ can be regarded as
a compact and strongly order preserving operator and, therefore, owing to the generalized version of the Krein--Rutman
theorem given in Section 6, the linear eigenvalue problem
\begin{equation}
\label{i.8}
  \left\{ \begin{array}{ll} \mathcal{P}u =\l u &\quad   \hbox{in}\;\;Q_T, \\
  \mathfrak{B}u = 0 &\quad\hbox{on}\;\;\p\O\times [0,T], \end{array}\right.
\end{equation}
possesses a unique principal eigenvalue, $\l_1[\mathcal{P},\mathfrak{B},Q_T]$,
associated with a positive eigenfunction
\begin{equation*}
  \v \in E_{\mathfrak{B}}\equiv \{u\in E\;:\;\; \mathfrak{B}u=0\}.
\end{equation*}
Moreover, adapting the arguments of \cite[Ch. 7]{LG13}, the principal eigenvalue is algebraically simple and strictly dominant. Although the existence of the principal eigenvalue is a classical result attributable to A. Beltramo and
P. Hess \cite{BH} in the special case when $\b \geq 0$, the corresponding existence result might be knew in the general setting of this paper, where $\b$ can change sign.
\par
Since the principal eigenfunction itself provides us with a positive strict supersolution
of $(\mathcal{P},\mathfrak{B},Q_T)$ if  $\l_1[\mc{P},\mf{B},Q_T]>0$, the next generalized scalar counterpart of \cite[Th. 2.2]{ALG} holds. Note that, in \cite{ALG},  the authors dealt with the special case when $\G_1=\emptyset$.

\begin{theorem}
\label{th12}
The following conditions are equivalent:
\begin{itemize}
\item[(a)] $\l_1[\mc{P},\mf{B},Q_T]>0$.

\item[(b)] $(\mathcal{P},\mathfrak{B},Q_T)$ possesses  a non-negative strict supersolution $h\in E\setminus\{0\}$.

\item[(c)]  Any strict supersolution $u\in E$ of
$(\mathcal{P},\mathfrak{B},Q_T)$ satisfies Alternative (b) of Theorem \ref{th11}, i.e., $(\mathcal{P},\mathfrak{B},Q_T)$ satisfies the strong maximum principle.
\end{itemize}
\end{theorem}

This theorem is an important periodic-parabolic counterpart of Theorem 2.4 of
H. Amann and J. L\'{o}pez-G\'{o}mez \cite{AL98}.
\par
The second, and main, goal of this paper is using Theorem \ref{th12} for sharpening   \cite[Th. 2.2]{ALG} up to cover the case of general cooperative systems of
periodic-parabolic type under arbitrary mixed boundary conditions for each of the underlying components. Our theorem provides us, as very special cases, with the elliptic counterparts of the main theorems of D. G. Figueiredo and E. Mittidieri \cite{FM}, G. Sweers \cite{Sw} and J. L\'opez-G\'omez and M. Molina-Meyer \cite{LM}, and should have a large number of
applications in the context of cooperative and quasi-cooperative periodic-parabolic systems.
\par
This paper is distributed as follows. Section 2 collects some classical results on the minimum principle for periodic-parabolic problems, which have been borrowed from the book of P. Hess \cite{Hess}, and derives  from them some extremely useful properties that will be used throughout this paper; in particular, the periodic-parabolic counterpart of the
generalized minimum principle of  M. H. Protter and H. F. Weinberger \cite{PW67}, which is one of the main findings
of this paper. Based on these results, Section 3 provides us with all the admissible behaviors of the supersolutions of $(\mathcal{P},\mathfrak{B},Q_T)$ in the presence of a positive supersolution bounded away from zero. Precisely, it shows the validity of Theorem \ref{th11} in the special case when $h(x,t)>0$ for all $(x,t)\in\bar Q_T$. Section 4 establishes an universal estimate for the  decaying rate  of the positive supersolutions of
$(\mathcal{P},\mathfrak{B},Q_T)$ along $\G_0$, which is a substantial extension of \cite[Th. 2.3]{LG13}. It is  necessary to complete the proof of Theorem \ref{th11} in Section 5 from the first classification theorem established in Section 3.  Section 6
sharpens substantially the main theorem of the authors in \cite{ALG}, by extending it to deal with arbitrary boundary conditions of
mixed type. Later, inspired on the  work of
S. Cano-Casanova and J. L\'opez-G\'omez \cite{CL-02}, Section 7 derives from the main theorem of Section 6 the most
fundamental properties of the principal eigenvalues of the cooperative  periodic-parabolic system introduced in Section 6.
Finally, in Section 8 the abstract theory of M. A. Krasnoselskij \cite{Kr} is invoked to characterize the principal eigenvalue of a cooperative periodic-parabolic system in a special case of great interest from the point of view of the applications.
\par
\vspace{0.2cm}
\noindent\textbf{On this paper:} This paper grew
from a preliminar version of Theorem \ref{th11} in \cite{LG14}. One and a half year after submission to the journal,
on July 18, 2016, R. Aftabizadeh knowledged the second author that he was going to ask for a technical report  to Juncheng Wei, who was the  handling editor of \cite{LG14}. There was not any technical report, nor any further news concerning that submission since then.
\par
In the mean time, the paper was completed by polishing, very substantially, the materials of Sections 2-5 and adding Sections 6 and 7. Then, the enlarged paper was submitted to J. Mallet-Paret for the \emph{Journal of Differential Equations} on July 19, 2016. After six months, on January 12, 2017,  W. M. Ni sent the authors a (positive) technical report. Three weeks later, on February 1, 2017, the authors sent back to the editorial  office of the JDE the revised manuscript following scrupulously the reviewer recommendations. Finally, the paper was rejected by W. M. Ni on May 2, 2017. It is the first time, and unique,  having more than 170 papers already published, that the second author deserves a rejection after sending back a revised version of a paper to the editorial office.
\par
Then, the paper was submitted to the \emph{Mathematische Annalen} on June 12, 2017. After less than two months, on August 1, 2017, the paper was rejected with an extremely bias report focusing all
the attention on the first 5 sections of the paper and forgetting about its contents for systems. According to
the handling editor, in this occasion Y. Giga, the reviewer (anonymous) had been suggested by H. Amann. Once Y. Giga realized that the technical report had been indeed bias, he proposed the authors to submit their paper to the \emph{Advances in Differential Equations}, where it was accepted on October 26, 2017, with a very positive technical report.
\par
Astonishingly, after four additional months, on February 22, 2018, three years and four months after \cite{LG14}
was submitted to the \emph{Differential and Integral Equations}, R. Aftabizadeh knowledged the second author that, since his university was not paying the subscription  to the \emph{Advances in Differential Equations}, he must pay 2200 USD for publishing the paper, and that \emph{with these difficulties, I must inform you that expected publication of your paper is at least two years after the date it was accepted}.

\setcounter{equation}{0}

\section{Classical Minimum principles}

\noindent In this section we collect some classical results on the minimum principle for periodic-parabolic problems that have been borrowed from the book of P. Hess \cite{Hess}. Then, we derive from them some extremely useful properties
that will be used throughout the rest of this paper. The next result is Proposition 13.1 of P. Hess \cite{Hess}. It is a very old result going back to L. Nirenberg \cite{Nirenberg}, which extends the classical minimum principle of E. Hopf \cite{Hopf}.

\begin{theorem}
\label{th2.1} Let  $G$ be a bounded domain of $\R^N\times\R$ and $\mf{L}(x,t)$ a uniformly elliptic operator in $G$
of the form
\begin{equation*}
  \mathfrak{L}:=\mathfrak{L}(x,t):= -\sum_{i,j=1}^N a_{ij}(x,t)\frac{\p^2}
  {\p x_i \p x_j}+\sum_{j=1}^N b_j(x,t) \frac{\p}{\p x_j}+c(x,t)
\end{equation*}
with $a_{ij}=a_{ji}, b_j, c \in \mc{C}^{\t,\frac{\t}{2}}(\bar G)$ and $c\geq 0$.
\par
Suppose a function $u \in \mc{C}^{2,1}(G)\cap\mc{C}(\bar G) $ satisfies
\begin{equation}
\label{ii.1}
   \p_t u(x,t) + \mathfrak{L}(x,t) u(x,t) \geq 0 \quad \hbox{for all} \quad (x,t)\in G,
\end{equation}
i.e., $u$ is super-harmonic for the parabolic operator $\mc{P}:=\p_t+\mf{L}(x,t)$,
and
\begin{equation*}
  m := \min_{\bar G} u\leq 0.
\end{equation*}
Assume that the minimum, $m$, is attained at an (interior) point $(x_0,t_0) \in G$. Then:
\begin{enumerate}
\item[{\rm (a)}]  $u=m$ in the (connected) component of
\[
  G(t_0):=\{(x,t)\in G\;:\;t=t_0\}
\]
containing $(x_0,t_0)$.

\item[{\rm (b)}] $u(x,t)=m$ if $(x,t)\in G$ can be connected with   $(x_0,t_0)$ by a path in $G$ consisting only of horizontal and upward vertical segments.
\end{enumerate}
\end{theorem}

\begin{figure}[h!]
\centerline{\includegraphics[scale=1]{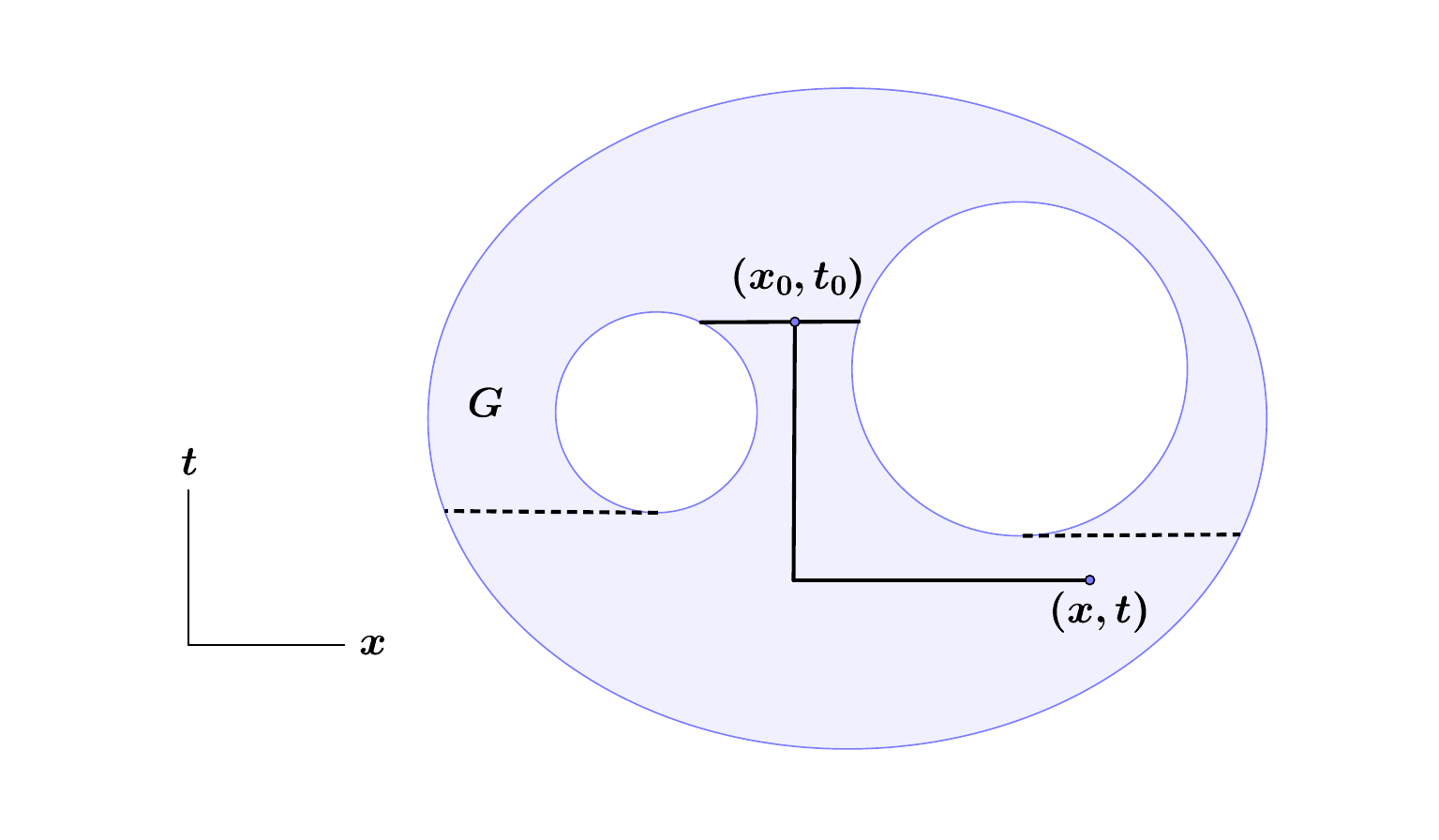}}
\caption{Construction of the portion of $G$ where $u\equiv m$.}
\label{ALG-Fig1}
\end{figure}

As a by-product of Theorem \ref{th2.1}, using $G= \O \times (0,T)$, the next result holds.

\begin{corollary}
\label{co2.1} Suppose $\O$ and $\mf{L}$ satisfy Assumptions {\rm (A1)} and {\rm (A2)} and, in addition, $c(x,t)\geq 0$ for all
$(x,t)\in\O \times \R$. Let $u \in E$ be a function such that
\begin{equation*}
  \mathcal{P} u \equiv \p_t u + \mf{L}u \geq 0  \quad \hbox{in}\;\; Q_T=\O \times [0,T]
\end{equation*}
and  $m=\min_{\bar Q_T} u \leq 0$. Then, $m$ cannot be reached in $\O\times (0,T)$, unless $u\equiv m$ in $\bar \O\times \R$.
\end{corollary}
\begin{proof}
Suppose there is $(x_0,t_0)\in \O \times (0,T)$ such that
\[
  u(x_0,t_0)=m\leq 0.
\]
Then, by Theorem \ref{th2.1} applied in $G:=\O\times (0,T)$, it follows that
$u(x,t)=m$ for $x\in \bar \O$ and $0\leq t \leq t_0$, because
$G(t_0)=\O$ and we are assuming $\O$ to be connected. Thus, since $u$ is $T$-periodic,
\[
  u(x,T)=u(x,0)=m
\]
for all $x \in \bar\O$. Therefore, applying again Theorem \ref{th2.1} yields $u \equiv m$ in $\bar \O\times [0,T]$. By
the time periodicity, $u\equiv m$ in $\bar\O \times \R$.
\end{proof}

The next result is the parabolic counterpart of the Hopf--Oleinik boundary lemma (see, e.g., Theorem 1.3 of \cite{LG13}). It has been
borrowed from Proposition 13.3 of P. Hess \cite{Hess}. The parabolic counterpart is attributable to L. Nirenberg \cite{Nirenberg}.

\begin{theorem}
\label{th2.2} Let $G$ and $\mf{L}$ be as in Theorem \ref{th2.1}. Suppose $u\in \mathcal{C}^{2,1}(G)\cap\mathcal{C}(\bar G)$ satisfies $\mc{P}u\geq 0$ in $G$ and there exists $P=(x_0,t_0)\in \partial G$ such that
\begin{equation*}
  u(x_0,t_0)=m := \min_{\bar G} u \leq 0.
\end{equation*}
Set
\begin{equation*}
  G_{t_0}:= \left\{(x,t)\in G \; : \;\; t \leq t_0\right\}
\end{equation*}
and assume that there exists an open ball  $B$, tangential to $\partial G$ at $P$,  such that
\begin{equation*}
  B_{t_0}:= \left\{(x,t)\in B\;  : \;\; t \leq t_0\right\}\subset G_{t_0}
\end{equation*}
and
\[
  u > m \quad \hbox{in}\;\; G_{t_0}.
\]
Further assume that the radial direction from the center of $B$ to $P$  is not parallel to the $t$-axis. Let $\nu$ be a direction in
$P$ pointing outward of $G_{t_0}$, having nonnegative $t$-component such that $-\nu$ points in $B_{t_0}$. Then,
\[
   \frac{\p u}{\p \nu}(x_0,t_0)<0.
\]
\end{theorem}

\begin{figure}[h!]
\centerline{\includegraphics[scale=1]{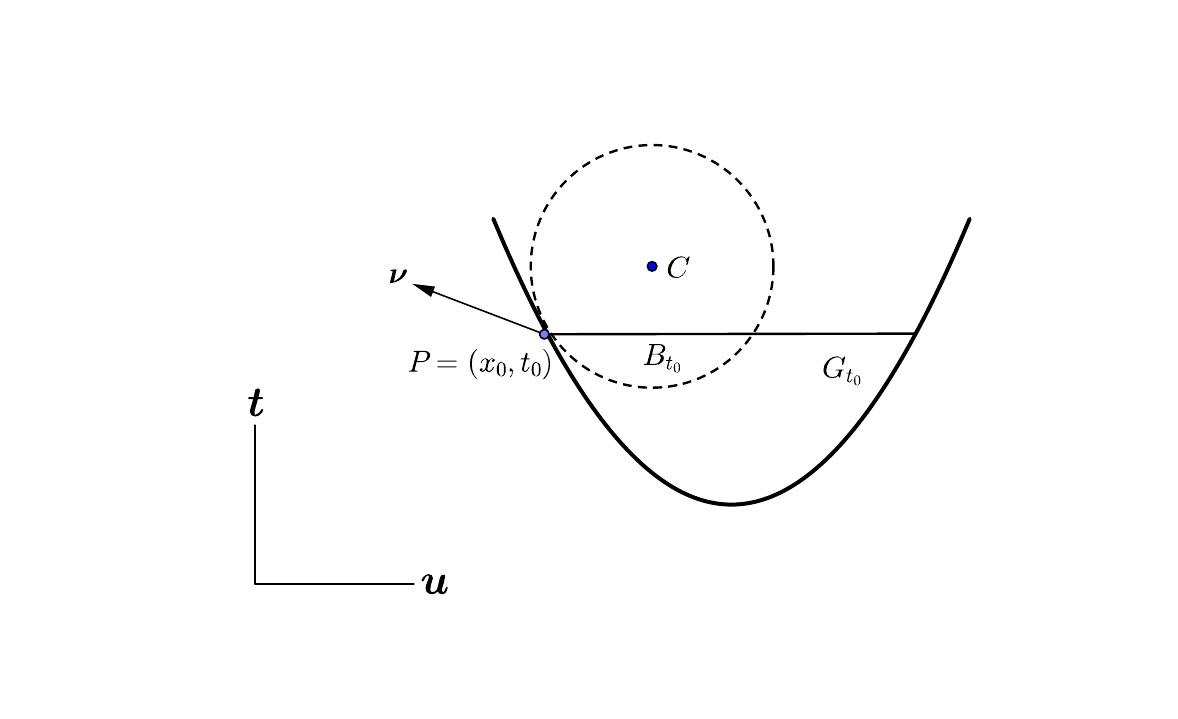}}
\caption{The parabolic boundary lemma.}
\label{ALG-Fig2}
\end{figure}

\begin{corollary}
\label{co2.2} Under the conditions of Corollary \ref{co2.1}, suppose $u$ is non-constant ($u>m$ in $\O \times \R$) and let
$(x_0,t_0)\in \p\O \times \R$ such that $u(x_0,t_0)=m$. Then, $\frac{\p u}{\p \nu}(x_0,t_0)<0$ for all outward pointing vector
field of $\O \times \R$ having nonnegative $t$-component.
\end{corollary}
\begin{proof}
By assumption (A1), $\O$ is sufficiently smooth as to guarantee the existence of an interior sphere property for the parabolic
cylinder. The rest of the proof is a direct consequence from Corollary \ref{co2.1} and Theorem \ref{th2.2}.
\end{proof}

In the more general case when $c$ can take negative values, one can apply the next periodic-parabolic counterpart of Theorem 10 of
Chapter 2 of Protter and Weinberger \cite{PW67} (see Theorem 1.7 of \cite{LG13}). In this result, instead of imposing $c\geq 0$,
as in the statement of Theorem \ref{th2.1}, we are assuming the existence of a super-harmonic function, $h$, everywhere positive.

\begin{theorem}
\label{th2.3}
 Suppose $\O$ and $\mf{L}$ satisfy Assumptions {\rm (A1)} and {\rm (A2)} and there exists $h \in E$, with
 $h(x,t)>0$ for all $(x,t)\in \bar Q_T$, such that
\begin{equation}
\label{ii.2}
  \mathcal{P} h \equiv \p_t h + \mf{L}h \geq 0  \quad \hbox{in}\;\; Q_T=\O \times [0,T].
\end{equation}
Then, for every $u \in E$ such that $\min_{\bar Q_T}u \leq 0$ and $\mc{P}u\geq 0$ in $Q_T$, the following holds
\begin{equation}
\label{ii.3}
  \frac{u(x,t)}{h(x,t)}> \min_{\bar Q_T}\frac{u}{h},\qquad (x,t)\in Q_T,
\end{equation}
unless $u/h$ is constant in $Q_T$.
\end{theorem}

When $c\geq 0$, one may choose $h\equiv 1$, because $\mc{P}1 = c(x,t)\geq 0$. Thus, Theorem \ref{th2.1} holds from Theorem \ref{th2.3}.

\begin{proof}
Set
\begin{equation*}
  v:= \frac{u}{h}\in E.
\end{equation*}
Then, differentiating and rearranging terms yields
\begin{equation*}
  \mathcal{P} u  = \mathcal{P}(hv)= \p_t(hv)+\mathfrak{L}(hv) =
  h \p_t v + v \p_t h + h \mathfrak{M}_h v,
\end{equation*}
where
\begin{equation*}
   \mathfrak{M}_h := -\sum_{i,j=1}^N a_{ij}  \frac{\p^2 }{\p x_i \p x_j}+
   \sum_{j=1}^N \left( b_j-\frac{2}{h}\sum_{i=1}^N a_{ij}
   \frac{\p h}{\p x_i}\right) \frac{\p }{\p x_j}
   +\frac{\mathfrak{L}h}{h}.
\end{equation*}
Since $h\in E$ and $h(x,t)>0$ for all $(x,t)\in\bar Q_T$, it becomes apparent that
\begin{equation*}
  \frac{\mathfrak{L}h}{h}\in F, \qquad   b_{j,h} := b_j-\frac{2}{h}\sum_{i=1}^N a_{ij}   \frac{\p h}{\p x_i}
  \in F,\qquad 1 \leq j \leq N.
\end{equation*}
So, $\mathfrak{M}_h$ satisfies the same requirements as $\mathfrak{L}$. Thus, setting
\begin{equation}
\label{ii.4}
   \mathfrak{L}_h := -\sum_{i,j=1}^N a_{ij}  \frac{\p^2 }{\p x_i \p x_j}+
   \sum_{j=1}^N b_{j,h} \frac{\p }{\p x_j}    +\frac{\p_t h + \mathfrak{L} h}{h}=\mf{M}_h+\frac{\p_t h}{h},
\end{equation}
and rearranging terms, we find that
\begin{equation*}
  \mathcal{P}u = h \left( \p_t + \mathfrak{L}_h \right)v.
\end{equation*}
Since, by \eqref{ii.2},
\begin{equation}
\label{ii.5}
   c_h:=\frac{\p_t h + \mathfrak{L}h}{h} = \frac{\mathcal{P}h}{h} \geq 0
\end{equation}
and $c_h \in F$, it becomes apparent that the periodic-parabolic operator
\begin{equation}
\label{ii.6}
  \mathcal{P}_h :=\p_t + \mathfrak{L}_h
\end{equation}
satisfies  the requirements to apply Theorem \ref{th2.1} to the function $v$, since $\mc{P}u\geq 0$ implies $\mc{P}_hv\geq 0$. Note that $\min_{\bar Q_T}u\leq 0$ implies $\min_{\bar Q_T}v\leq 0$.
\end{proof}

\section{The first classification theorem}

\noindent The following result provides us with all the admissible
behaviors of the supersolutions of $(\mathcal{P},\mathfrak{B},Q_T)$ in the presence of a positive supersolution bounded away from zero.
It is a periodic-parabolic counterpart of Theorem 1 of W. Walter \cite{Wa89} and the first classification theorems of
J. L\'opez-G\'omez \cite{LG03,LG13}.

\begin{theorem}
\label{th3.1} Suppose $(\mathcal{P},\mathfrak{B},Q_T)$ admits a supersolution $h \in E$ such that
\begin{equation}
\label{3.1}
  h(x,t)>0 \qquad \hbox{for all}\quad (x,t)\in \bar Q_T.
\end{equation}
Then, any supersolution $u\in E$  of $(\mathcal{P},\mathfrak{B},Q_T)$ satisfies one of the following alternatives:
\begin{enumerate}
\item[{\rm (a)}] $u =  0$ in $Q_T$.

\item[{\rm (b)}] $u(\cdot,t)\gg 0$ for all $t \in [0,T]$, in the sense that $u(x,t)>0$ for all $x\in\O\cup\G_1$,  $t \in [0,T]$,  and
\begin{equation*}
  \frac{\p u}{\p \nu}(x,t)<0\quad \hbox{for all}\;\; (x,t) \in [u^{-1}(0)\cap \G_0]\times [0,T].
\end{equation*}

\item[{\rm (c)}] There exists $m<0$ such that
$u = m h$ in $Q_T$. In such case, $\G_0=\emptyset$ and
\begin{equation}
\label{3.2}
  \left\{ \begin{array}{ll} \mathcal{P}h =0 &\quad   \hbox{in}\;\;Q_T, \\  \mathfrak{B}h = 0
  &\quad\hbox{on}\;\;\p\O\times [0,T].\end{array}\right.
\end{equation}
\end{enumerate}
\end{theorem}

\noindent Note that \eqref{3.2} entails $\tau=0$ to be an eigenvalue of  the linear problem
\begin{equation}
\label{3.3}
  \left\{ \begin{array}{ll} \mathcal{P}\varphi =\tau \varphi  &\quad
  \hbox{in}\;\;Q_T, \\ \mathfrak{B}\varphi = 0
  &\quad\hbox{on}\;\;\p\O\times [0,T],\end{array}\right.
\end{equation}
with associated eigenfunction $\varphi=h$.  In particular,  $\tau=0$ is an eigenvalue to a positive
eigenfunction of $(\mc{P},\mf{B},Q_T)$ if Alternative (c) holds.

\begin{proof} Throughout this proof, we will use the notations introduced in the proof of
Theorem \ref{th2.3}. Let  $u\in E$ be an arbitrary  supersolution of $(\mc{P},\mathfrak{B},Q_T)$. Then,
\begin{equation*}
  v:= \frac{u}{h}\in E.
\end{equation*}
Moreover,
\begin{equation}
\label{3.4}
  \mathcal{P}u =\mathcal{P}(hv)= h \mc{P}_h v,
\end{equation}
where
\[
  \mc{P}_h=\p_t +\mf{L}_h
\]
with $\mf{L}_h$  given by \eqref{ii.4}. Hence, since we are assuming that $\mathcal{P}u \geq 0$ in $Q_T$, we find that
\begin{equation}
\label{3.5}
  \mathcal{P}_h v \geq 0\qquad \hbox{in}\quad Q_T.
\end{equation}
On the other hand, along $\G_0$ we have that $u \geq 0$. Thus, since $h>0$, we also have that
\begin{equation}
\label{3.6}
  v \geq 0\qquad \hbox{on}\quad \G_0\times [0,T].
\end{equation}
Finally, along $\G_1\times [0,T]$, we obtain that
\begin{align*}
   0 & \leq \mathfrak{B} u = \mathfrak{B} (hv)=h \frac{\p v}{\p \nu} + v \frac{\p h}{\p \nu}+\b hv
   = h\left( \frac{\p v}{\p \nu}+\frac{\frac{\p h}{\p \nu} +\b h}{h}v \right) \cr & =
   h \left( \frac{\p v}{\p \nu}+\frac{\mathfrak{B}h}{h}v\right) =
   h \left( \frac{\p v}{\p \nu}+\b_h v\right),
\end{align*}
where we have denoted
\begin{equation}
\label{3.7}
 \b_h := \frac{\mathfrak{B}h}{h}\geq 0\qquad \hbox{on}\quad \G_1\times [0,T],
\end{equation}
because $\mathfrak{B}h \geq 0$ on $\p\O\times [0,T]$. Consequently,
\begin{equation}
\label{3.8}
  0 \leq \mathfrak{B} u = h \mathfrak{B}_h v \qquad \hbox{on}\quad   \p\O\times [0,T],
\end{equation}
where $\mathfrak{B}_h$ stands for the boundary operator
\begin{equation}
\label{3.9}
    \mathfrak{B}_h \xi := \left\{ \begin{array}{ll}
    \xi \qquad & \hbox{on } \;\;\G_0\times [0,T] \\     \frac{\p \xi}{\p \nu} + \b_h \xi \qquad &
   \hbox{on } \;\; \G_1\times [0,T]  \end{array} \right.
\end{equation}
for all $\xi \in \mathcal{C}(\G_0)\oplus \mathcal{C}^1(\O\cup \G_1)$. Summarizing, the function $v$ is a supersolution of
$(\mathcal{P}_h,\mathfrak{B}_h,Q_T)$. Subsequently, we will distinguish  three different cases.
\par
\vspace{0.2cm}
\noindent\textbf{Case 1:} Suppose $u(x_0,t_0)<0$ for some $x_0\in\O$ and $t_0\in [0,T]$. Then,
\[
  v(x_0,t_0)=\tfrac{u(x_0,t_0)}{h(x_0,t_0)}<0
\]
and hence,
\begin{equation*}
  m:= \min_{\bar Q_T} v <0.
\end{equation*}
Thus, since $c_h\geq 0$ and $\mathcal{P}_h v\geq 0$ in $ Q_T$, it follows from Corollary \ref{co2.1} that either
\begin{equation}
\label{3.10}
  v(x,t)>m \qquad \hbox{for all}\quad (x,t)\in\O\times [0,T],
\end{equation}
or
\begin{equation}
\label{3.11}
  v\equiv  m \qquad \hbox{in}\quad \bar Q_T.
\end{equation}
Suppose  \eqref{3.10} holds. Then, thanks to \eqref{3.6},
\begin{equation*}
  v(x,t)>m \qquad \hbox{for all}\quad (x,t)\in (\O\cup \G_0)\times [0,T].
\end{equation*}
Consequently, there exist $x_1\in \G_1$ and $t_1 \in [0,T]$ such that $m=v(x_1,t_1)$. Moreover, by Corollary \ref{co2.2}, $\frac{\p v}{\p\nu}(x_1,t_1)<0$. Hence,
\begin{align*}
  0 \leq \mathfrak{B}_h v (x_1,t_1) & = \frac{\p v}{\p
  \nu}(x_1,t_1)+\b_h(x_1,t_1) v(x_1,t_1)\\ & < \b_h(x_1,t_1) v(x_1,t_1) =
  \b_h(x_1,t_1) m,
\end{align*}
which implies $\b_h(x_1,t_1)<0$ because $m<0$. By \eqref{3.7}, this is impossible. Consequently, instead of \eqref{3.10}, condition \eqref{3.11}   holds. This implies $u = m h$ in $\bar Q_T$ and Alternative (c) occurs. The remaining assertions of Alternative (c) can be inferred as follows. Suppose $\G_0\neq \emptyset$ and pick $x_0\in \G_0$. Then, since
$h(x_0,t)>0$ for all $t \in [0,T]$ and $m<0$, we find that
\[
  u(x_0,t) = m h(x_0,t)<0,
\]
which is impossible,
because $u\geq 0$ on $\G_0\times [0,T]$. This contradiction shows that $\G_0=\emptyset$. Moreover,
\begin{equation*}
  0\leq \mathcal{P} u = m \mathcal{P}h  \leq 0
\end{equation*}
implies  $\mathcal{P}u=0$ in $\O\times [0,T]$ and
\begin{equation*}
  0 \leq \mathfrak{B} u = m \mathfrak{B}h \leq 0
\end{equation*}
entails $\mathfrak{B} u = 0$ on $\p\O\times [0,T]$. Consequently, \eqref{3.2} holds.
\par
To complete the proof of the theorem it remains to show that some of the first two alternatives occurs when
$u\geq 0$ in $Q_T$.
\par
\vspace{0.2cm}
\noindent\textbf{Case 2:} Suppose $u\geq 0$ and $u(x_0,t_0)=0$ for some $(x_0,t_0)\in\O\times [0,T]$. Then,
$v\geq 0$ in $\bar Q_T$ and $v(x_0,t_0)=0$. Thus, since $v$ is a
supersolution of $(\mathcal{P}_h,\mathfrak{B}_h,Q_T)$, by Corollary \ref{co2.1}, $v=0$ in $\bar Q_T$.
Therefore, $u=0$ in $\bar Q_T$ and Alternative (a) holds.
\par
\vspace{0.2cm}
\noindent\textbf{Case 3:}
Finally, suppose that $u(x,t)>0$ for all $x\in\O$ and $t\in [0,T]$.
Then, $v(x,t)>0$ for all $x\in\O$ and $t\in [0,T]$. Moreover, since $v$ is a supersolution
of $(\mathcal{P}_h,\mathfrak{B}_h,Q_T)$, it follows from Corollary \ref{co2.2} that
\begin{equation}
\label{3.12}
  \frac{\p v}{\p \nu}(x,t) < 0\qquad \hbox{for all}\quad
  (x,t)\in \left[ v^{-1}(0)\cap \p\O\right] \times [0,T].
\end{equation}
Suppose $\G_1\neq \emptyset$ and $v(x_1,t)=0$ for some $x_1\in\G_1$ and $t\in [0,T]$. Then,
\begin{equation*}
  0 \leq \mathfrak{B}_h v(x_1,t)=\frac{\p v}{\p   \nu}(x_1,t)+\b_h(x_1,t)v(x_1,t)= \frac{\p v}{\p \nu}(x_1,t)
\end{equation*}
which contradicts \eqref{3.12}. Thus, $v(x_1,t)>0$ for all $x_1\in\G_1$ and $t\in [0,T]$. So,
also $u$ satisfies this property. Moreover, for each $x_0\in\G_0$ with $u(x_0,t)=0$ for some $t \in [0,T]$,
we have that $v(x_0,t)=0$ and hence \eqref{3.12} yields
\begin{align*}
  \frac{\p u}{\p \nu}(x_0,t)  & = \frac{\p (hv)}{\p \nu}(x_0,t)
   = h(x_0,t) \frac{\p v}{\p \nu}(x_0,t)+v(x_0,t) \frac{\p h}{\p
   \nu}(x_0,t) \cr & = h(x_0,t) \frac{\p v}{\p \nu}(x_0,t)< 0.
\end{align*}
Therefore, Alternative (b) holds. This ends the proof.
\end{proof}

As an immediate consequence of  Theorem \ref{th3.1}
the next results follow.

\begin{corollary}
\label{co31} Suppose $(\mathcal{P},\mathfrak{B},Q_T)$ admits a strict supersolution $h \in E$ satisfying \eqref{3.1}.
Then, every supersolution  $u\in E \setminus\{0\}$ of $(\mathcal{P},\mathfrak{B},Q_T)$ (in particular, any strict supersolution) satisfies $u(x,t)>0$ for all $(x,t) \in (\O\cup\G_1)\times [0,T]$, and
\begin{equation*}
  \frac{\p u}{\p \nu}(x,t)<0 \quad \hbox{for all} \;\; (x,t)\in \left[ u^{-1}(0)\cap \G_0\right]\times [0,T].
\end{equation*}
\end{corollary}

\begin{proof} Since $h$ is a strict supersolution of $(\mathcal{P},\mathfrak{B},Q_T)$, \eqref{3.2} cannot be satisfied and hence, Alternative (c) cannot occur. As $u\neq 0$,  Alternative (b) occurs.  \end{proof}

As a consequence of Corollary\,\ref{co31}, the following uniqueness result holds.

\begin{corollary}
\label{co32}
Suppose $(\mathcal{P},\mathfrak{B},Q_T)$ admits a strict supersolution $h \in E$ satisfying \eqref{3.1}.
Then, $u=0$ is the unique function $u\in E$ solving
\begin{equation}
\label{3.13}
  \left\{ \begin{array}{ll} \mathcal{P}u = 0 &\quad   \hbox{in}\;\;Q_T, \\ \mathfrak{B}u = 0
  &\quad\hbox{on}\;\;\p\O\times [0,T].\end{array}\right.
\end{equation}
Therefore, for every $f, g \in F$, the boundary value problem
\begin{equation}
\label{3.14}
  \left\{ \begin{array}{ll} \mathcal{P}u = f \quad &   \hbox{in}\;\; Q_T \\
  \mathfrak{B}u = g & \hbox{on}\;\;\p\O\times [0,T] \end{array} \right.
\end{equation}
possesses at most one solution $u\in E$.
\end{corollary}

\begin{proof} Suppose $u\in E\setminus\{0\}$ solves \eqref{3.13}. Then, by Corollary\,\ref{co31},
\begin{equation*}
  u(x,t)>0\quad \hbox{for all}\;\; (x,t) \in (\O\cup\G_1)\times [0,T].
\end{equation*}
Moreover, as  $-u\in E\setminus\{0\}$ provides us with another solution of \eqref{3.13}, we also have
\begin{equation*}
 -u(x,t)>0\quad \hbox{for all}\;\; (x,t) \in (\O\cup\G_1)\times [0,T],
\end{equation*}
which is impossible. Therefore, $u=0$ is the only function $u\in E$ solving \eqref{3.13}. Now, the uniqueness result for \eqref{3.14} is obvious.  \end{proof}

According to Lemma 2.1 of Chapter 2 of \cite{LG13},
there exist a function $\psi \in \mathcal{C}^{2+\t}(\bar\O)$ and a positive constant $\g >0$ such that
\begin{equation}
\label{3.15}
  \frac{\p \psi}{\p \nu}(x) \geq \g \qquad \hbox{for all}\quad x\in   \G_1.
\end{equation}
The extra H\"{o}lder regularity of $\psi$ is a consequence of the regularity of $\p\O$.

\begin{proposition}
\label{pr31} Let $\psi \in \mathcal{C}^{2+\t}(\bar\O)$ be a function satisfying \eqref{3.15} for some constant $\g>0$. Then, there exist $\o_0\in\R$ and $M>0$ such that
\begin{equation}
\label{3.16}
  h:= e^{M \psi}\in\mathcal{C}^{2+\t}(\bar\O)
\end{equation}
is a strict supersolution of $(\mathcal{P}+\o,\mathfrak{B},Q_T)$ for all $\o>\o_0$.
\end{proposition}

\begin{proof}  Suppose $h\in \mathcal{C}^{2}(\bar\O)$, with $h(x)>0$ for all $x\in \bar\O$, is a strict supersolution of $(\mathcal{P}+\o_0,\mathfrak{B},Q_T)$ for some $\o_0\in\R$.
Then, for every $\o>\o_0$,
\begin{equation*}
  (\mathcal{P}+\o)h=(\mathcal{P}+\o_0)h+(\o-\o_0)h \geq   (\o-\o_0)h>0.
\end{equation*}
Hence, $h$ provides us with a strict supersolution of $(\mathcal{P}+\o,\mathfrak{B},Q_T)$. Therefore, it suffices to show that  there exist $\o_0\in\R$ and $M>0$ such that \eqref{3.16} is a strict supersolution
of $(\mathcal{P}+\o_0,\mathfrak{B},Q_T)$. We are assuming $h(x)>0$ for each $x\in\bar\O$. Moreover, by \eqref{3.16}, along $\G_1$ we have that
\begin{equation*}
  \mathfrak{B} h  = M h \frac{\p \psi }{\p \nu} + \b h
    = \left( M \frac{\p \psi }{\p \nu}+\b\right) h\geq \left( M \g +\b\right) h >0
\end{equation*}
provided $M>0$ is sufficiently large. Therefore,
\begin{equation*}
  \mathfrak{B} h >0\qquad \hbox{on}\quad \p\O
\end{equation*}
for sufficiently large $M>0$. Moreover, since $h$ is independent on $t$ and it is positive and separated away from zero in $\bar \O$,
\begin{equation*}
  (\mathcal{P}+\o_0) h = (\p_t+\mathfrak{L}+\o_0)h= \mathfrak{L} h+\o_0 h >0
\end{equation*}
for sufficiently large $\o_0>0$. This ends the proof.
\end{proof}

As a by-product of Corollary \ref{co32}, Proposition \ref{pr31} and the abstract theory of H. Amann \cite{Amann}, P. Hess \cite{Hess} and D. Daners and P. Koch-Medina \cite{DK}, the next result holds.

\begin{theorem}
\label{th32} There exists $\o_0\in\R$ such that, for each $\o > \o_0$ and  $f\in F$, the periodic-parabolic problem
\begin{equation}
\label{3.17}
  \left\{ \begin{array}{ll}( \mathcal{P}+\o)u = f  \quad &   \hbox{in}\;\; Q_T, \\
  \mathfrak{B}u = 0 & \hbox{on}\;\;\p\O\times [0,T], \end{array} \right.
\end{equation}
possesses a unique solution $u\in E$. Moreover, if $f>0$, i.e., $f\geq 0$ but  $f\neq 0$, then $u \gg 0$ in the sense that
\begin{equation}
\label{3.18}
   u(x,t)>0 \qquad \hbox{for all}\quad (x,t)\in \left(\O\cup\G_1\right)\times [0,T],
\end{equation}
and
\begin{equation}
\label{3.19}
  \frac{\p u}{\p \nu}(x,t)<0 \qquad \hbox{for all}\quad   (x,t) \in \G_0 \times [0,T].
\end{equation}
\end{theorem}

\begin{proof}
By Proposition\,\ref{pr31}, there exist $\o_0\in\R$ and $h\in\mathcal{C}^{2+\t}(\bar\O)$ such that
$h(x)>0$ for all $x\in \bar\O$ and it is a strict supersolution of $(\mathcal{P}+\o,\mathfrak{B},Q_T)$ for
all $\o>\o_0$. The uniqueness assertion follows from Corollary\,\ref{co32}.
\par
Suppose $\o>\o_0$, $f\geq 0$, $f\neq 0$, and \eqref{3.17} has a solution $u\in E$. Then, $u$ is a strict supersolution of $(\mathcal{P}+\o,\mathfrak{B},Q_T)$ and, owing to Corollary\,\ref{co31}, \eqref{3.18} and \eqref{3.19} hold.
\par
In order to establish the existence we will assume, in addition, that
\[
  \max_{\bar Q_T} c+\o_0 > 0.
\]
If $\b\geq 0$ on $\G_1$, then the existence follows from the abstract theory of P. Hess \cite{Hess}. Suppose, more generally, that
$\b$ takes some negative value. Then, we perform the change of variable
\begin{equation*}
  u = hv,\qquad h:= e^{M\psi},
\end{equation*}
for a sufficiently large constant $M>0$. Using the notations introduced in the proof of Theorem \,\ref{3.1} we have that, for every $\o\in\R$,
\begin{equation*}
  (\mathcal{P}+\o)u=h(\mathcal{P}_h +\o)v\qquad \hbox{in}\quad Q_T.
\end{equation*}
Now, let us enlarge $\o_0$, if necessary, so that
\[
  \max_{\bar Q_T} c_h +\o_0 > 0.
\]
Along $\p\O$ we have that
\[
  \mathfrak{B}u = h \mathfrak{B}_h v,
\]
where $\mathfrak{B}_h$ is the Dirichlet operator on $\G_0$ and the oblique derivative operator
\[
  \mathfrak{B}_h := \p_\nu +\tfrac{\mathfrak{B}h}{h}
\]
on $\G_1$. As
\[
  \tfrac{\mathfrak{B}h}{h}= M \p_\nu \psi + \b,
\]
thanks to \eqref{3.15},  $\mathfrak{B}h/h>0$ on $\G_1$ for sufficiently large $M>0$. Therefore, applying Lemma 14.3 of
P. Hess \cite{Hess} to the transformed problem
\begin{equation*}
  \left\{ \begin{array}{ll}( \mathcal{P}_h+\o)v = \tfrac{f}{h}  \quad &
  \hbox{in}\;\; Q_T, \\    \mathfrak{B}_h v = 0  & \hbox{on}\;\;\p\O\times[0,T], \end{array} \right.
\end{equation*}
ends the proof. Although \cite{Hess} did not deal with mixed boundary conditions, the invertibility and positivity results of Section
13 of D. Daners and P. Koch-Medina \cite{DK}, applicable in our general setting here, allow us to adapt, very easily, the proof of  Lemma 14.3 of \cite{Hess} to conclude the proof of the theorem.  Once guaranteed that the underlying evolution operators are well defined (this task has been accomplished in \cite{DK}), the result follows easily from the formula of variation of the constants as in \cite{Hess}. The technical details are omitted here.
\end{proof}

\section{Uniform boundary lemma for periodic parabolic problems}

\noindent  The next result provides us with a uniform estimate of the decaying rate  of the positive supersolutions of
$(\mathcal{P},\mathfrak{B},Q_T)$ along $\G_0$. It is a substantial extension of Theorem 2.3 of J. L\'opez-G\'omez \cite{LG13}.

\begin{theorem}
\label{th41}
 Any non-negative supersolution $u\in E\setminus \{0\}$ of $(\mathcal{P},\mathfrak{B},Q_T)$ satisfies
\eqref{3.18} and \eqref{3.19}. Moreover, there exists $\d=\d(u)>0$ such that
\begin{equation}
\label{4.1}
  u(x,t) \geq \d \,\mathrm{dist\,}(x,\G_0)\qquad \hbox{for all}\quad (x,t)\in\O\times [0,T].
\end{equation}
\end{theorem}

\begin{proof} By Proposition \ref{pr31}, there exist $M>0$ and $\o>0$ such that $h:= e^{M\psi}$ is a positive strict
supersolution of $(\mathcal{P}+\o,\mathfrak{B},Q_T)$. As $u$ is a supersolution of $(\mathcal{P},\mathfrak{B},Q_T)$,
\begin{equation*}
  (\mathcal{P}+\o)u=\mathcal{P}u+\o u \geq \o u >0 \qquad   \hbox{in}\quad Q_T,
\end{equation*}
since $u\geq 0$ and $u\neq 0$  in $\O$. Moreover  $\mathfrak{B}u \geq 0$ on $\p\O$. Thus $u>0$ is a strict supersolution of
$(\mathcal{P}+\o,\mathfrak{B},Q_T)$ and Corollary \ref{co31} establishes \eqref{3.18}  and \eqref{3.19}.
\par
The uniform estimate \eqref{4.1} will follow from \eqref{3.18} and \eqref{3.19}, since $\p\O$ satisfies the uniform decaying property of Hopf along $\p\O$, as discussed in \cite{LG03} and \cite{LG13}. Its proof will be completed in two steps.
\par
First of all, we will  give a uniform estimate for $u(x,t)$  with $x\in B_R(x_0)\subset \O $  for every  $x_0\in \O $ and sufficiently small $R>0$ such that $B_R(x_0)\subset\O$ and $t\in [0,T]$. In such circumstances, we will show that there exists a constant
$M:=M(\mathfrak{L},R)>0$ such that, for every $u\in E$ satisfying
\begin{equation}
\label{4.2}
 \left\{ \begin{array}{ll}u(x,t)> 0 &\quad   x\in B_R (x_0), \quad  t\in [0,T], \\
   \mathcal{P}u(x,t) \geq  0 &\quad x\in B_R(x_0), \quad  t\in [0,T], \end{array}\right.
\end{equation}
the following estimate holds:
\begin{equation}
\label{4.3}
  u(x,t)\geq  \left(M \min_{(x,t)\in \bar B_\frac{R}{2}  (x_0)\times [0,T] } u\right) \mathrm{dist\,}
  (x,\partial B_R(x_0))
\end{equation}
for all $(x,t) \in \bar B_R (x_0)\times[0,T]$. Later, in the second step, we will prove that $\d$ can be chosen to be independent of $R$. \par
To show \eqref{4.3}, let $R>0$ and $x_0\in \Omega$ be such that $B_R (x_0)\subset \Omega$, suppose $u\in E$ satisfies \eqref{4.2} and  consider the auxiliary functions
\begin{equation*}
E(x):=e^{\alpha(R^2-|x-x_0|^2)}, \qquad  v(x):=E(x)-1, \qquad x\in \R^N,
\end{equation*}
where $\alpha >0$ is a constant to be chosen later. The function $v$ satisfies
\begin{equation*}
 \left\{ \begin{array}{lll}v(x)> 0 \quad \hbox{if and only if} &\quad   |x-x_0|< R, \\
   v(x)= 0 \quad \hbox{if and only if} &\quad   |x-x_0|= R,\\ v(x)< 0 \quad \hbox{if and only if} &\quad   |x-x_0|> R. \end{array}\right .
\end{equation*}
Subsequently, we set
\begin{equation*}
   D:=\left\{x\in \R^N : \quad \frac {R}{2}<|x-x_0|< R\right\},
\end{equation*}
\begin{equation*}
c^+:=\mathrm{max\,}  \left\{c,0\right\}\geq c, \quad  \mathfrak{L}^+:= \mathfrak{L}-c+c^+,\quad
\mc{P}^+:=\p_t+\mf{L}^+.
\end{equation*}
Then, since $v(x)$ has been taken to be independent of $t$,
\begin{equation*}
  \mathcal{P^+} v  = \frac {\partial v}{\partial t}+\mathfrak{L^+}v = \mathfrak{L^+}v.
\end{equation*}
Note that
\begin{align*}
   \mathfrak{L^+}v & := -\sum_{i,j=1}^N a_{ij}(x,t) \frac{\partial^2 E }{\partial x_i \p x_j}(x)+
   \sum_{j=1}^N b_j(x,t) \frac{\partial E}{\partial x_j}(x)+ c^+ (x,t)v\\[1ex]
   & = -\sum_{i,j=1}^N a_{ij}(x,t) \frac{\partial^2 E }{\partial x_i \p x_j}(x)\!+\!
   \sum_{j=1}^N b_j(x,t) \frac{\partial E}{\partial x_j}(x)\!+\! c^+ (x,t)E(x)\!-\! c^+(x,t)\\[1ex]
   & = \mf{L}^+E(x)-c^+(x,t).
\end{align*}
Moreover, for every $j\in \left\{1,...,N\right\}$,
\begin{equation*}
\frac{\partial E}{\partial x_j}(x)=-2\alpha (x_j-x_{0j}) e^{\alpha (R^2-|x-x_0|^2)},
\end{equation*}
 \begin{equation*}
\frac{\partial^2 E}{\partial x_j^2}(x)=[4\alpha^2 (x_j-x_{0j})^2-2\alpha] e^{\alpha(R^2-|x-x_0|^2)},
\end{equation*}
while, for every $i, j\in \left\{1,...,N\right\}$,  with $i\neq j$, and $x\in R^N$,
\begin{equation*}
\frac{\partial^2 E}{\partial x_i \p xj}(x)=4\alpha^2 (x_i-x_{0i})(x_j-x_{0j}) e^{\alpha(R^2-|x-x_0|^2)}.
\end{equation*}
Thus, rearranging terms yields
\begin{align*}
   \mathcal{P^+}v & = \Big\{ -4\alpha^2\sum_{i,j=1}^N a_{ij}(x_i-x_{0i})(x_j-x_{0j})+
  2\alpha \sum_{j=1}^N [a_{jj}-b_j(x_j-x_{0j})] \cr &\hspace{7cm} + c^+ (x,t)\Big \}E(x)-c^+(x,t) \cr &
  =\Big\{ -4\alpha^2 (x-x_0)^T A_{(x,t)}(x-x_0)+2\alpha \left[\mathrm{tr\,} A_{(x,t)}-\langle b(x,t),x-x_0\rangle \right]
  \cr & \hspace{7cm} + c^+ (x,t)\Big\}E(x)-c^+(x,t),
\end{align*}
where $A_{(x,t)}$ is the matrix of the principal coefficients of $ \mathfrak{L}$, $\mathrm{tr\,}A_{(x,t)}$ stands for the trace of $A_{(x,t)}$ and $b:=(b_1,...,b_N)$. Let $\mu >0$ be the ellipticity constant of $ \mathfrak{L}(x,t)$ in $\O \times [0,T]$. Then, for every $x\in D$ and  $t\in [0,T]$
\begin{equation*}
   (x-x_0)^T A_{(x,t)}(x-x_0)\geq \mu |x-x_0|^2 \geq \mu \frac{R^2}{4}.
\end{equation*}
Moreover,
\begin{align*}
\left|\mathrm{tr\,}A_{(x,t)}-\langle b(x,t),x-x_0\rangle \right| & \leq \left|\mathrm{tr\,} A_{(x,t)}\right|+|b(x,t)||x-x_0|\\ & \leq \left|\mathrm{tr\,} A_{(x,t)}\right|+|b(x,t)|R.
\end{align*}
Since $a_{ij}, b_j, c^+\in F$, $1\leq i, j \leq N$, there exists a constant
\begin{equation*}
C:=C(a_{jj},b_j,R)=C(\mathfrak{L},R)>0
\end{equation*}
such that
\begin{equation*}
\left|\mathrm{tr\,}A_{(x,t)}-\langle b(x,t),x-x_0\rangle \right|\leq C \quad \hbox{for all} \;\;  (x,t)\in D \times [0,T],
\end{equation*}
independently of $x_0$. Thus, for every $x\in D$,
\begin{equation*}
 \mathcal{P^+} v \leq ( -\alpha^2 \mu R^2 + 2\alpha C+c^+)E(x)-c^+(x,t)
\end{equation*}
and hence, there exists $\alpha :=\alpha(\mathfrak{L},R)>0$ such that
\begin{equation}
\label{4.4}
 \mathcal{P^+}v \leq 0 \quad \hbox{for all} \quad (x,t)\in D\times [0,T].
\end{equation}
By  \eqref{4.2}, $u(x,t)>0$ for all $ x\in \bar{B}_{\frac{R}{2}}(x_0)$ and $t\in [0,T]$. Thus, by the continuity of $u$,
\begin{equation*}
 u_R:=\min_{(x,t)\in \bar {B}_{\frac{R}{2}}(x_0)\times [0,T]}u>0.
\end{equation*}
Consider  the auxiliary function
\begin{equation}
\label{4.5}
   w(x,t) :=u(x,t)-\e v(x), \qquad \e := \frac {u_R}{e^{\a R^2}-1},
\end{equation}
which is well defined in $\bar {B}_R(x_0)\times [0,T]$. Since
\begin{equation*}
   u(x,t) \geq u_R  \quad \hbox{and} \quad v\leq v(x_0)=e^{\a R^2}-1 \quad \hbox{in}\;\;
   \bar{B}_\frac{R}{2}(x_0)\times[0,T],
\end{equation*}
it follows from \eqref{4.5} that
\begin{equation*}
    w \geq u_R-\e (e^{\a R^2}-1)=0 \quad \hbox{in} \quad \bar{B}_\frac{R}{2}(x_0)\times[0,T].
\end{equation*}
In particular,
\begin{equation*}
    w(x,t) \geq 0  \qquad \hbox{on} \qquad \partial B_\frac{R}{2}(x_0)\times [0,T].
\end{equation*}
Moreover, since
\begin{equation*}
   v = 0 \quad \hbox{on}\quad \partial B_R(x_0)\quad \hbox{and}\quad u\geq 0 \quad\hbox{in} \quad       \bar {B}_R(x_0)\times [0,T],
\end{equation*}
we also have that
\begin{equation*}
   w(x,t) =u(x,t) \geq 0  \qquad \hbox{on} \qquad \p B_R(x_0)\times [0,T].
\end{equation*}
Summarizing,
\begin{equation}
\label{4.6}
  w(x,t) \geq 0  \quad \hbox{for all} \quad x \in \p D=\partial B_R(x_0)\cup \partial B_\frac{R}{2}(x_0)\quad
  \hbox{and} \quad t\in [0,T].
\end{equation}
Moreover, by  \eqref{4.4},
\begin{equation*}
 \mathcal{P}^+ w =\mathcal{P}^+ u-\e \mathcal{P}^+ v\geq \mathcal{P}^+ u =\mc{P}u +(c^+-c)u.
\end{equation*}
Hence, by \eqref{4.2}, it becomes apparent that
\begin{equation}
\label{4.7}
 \mathcal{P^+} w \geq (c^+-c)u\geq 0 \quad \hbox{in} \quad D\times[0,T],
\end{equation}
because $c^+\geq c$. Subsequently, we will prove that all these properties entail
\begin{equation}
\label{4.8}
  w(x,t) :=u(x,t)-\e v\geq 0 \quad \hbox{in} \quad \bar D\times[0,T].
\end{equation}
Indeed, either $w (x,t)>0$ for all $(x,t)\in \bar D\times [0,T]$, and then \eqref{4.8} holds, or
\begin{equation*}
   m:=\min_{\bar D\times [0,T]}w \leq 0.
\end{equation*}
In such case, by \eqref{4.7}, we may infer from Corollary \ref{co2.1} that $m$ cannot be reached in
$D\times (0,T)$, unless $w\equiv m\leq 0$ in $\bar D\times [0,T]$.
\par
When $m$ is not reached in $D\times (0,T)$, then it must be reached on the boundary and, thanks to \eqref{4.6}, $m=0$ and hence, \eqref{4.8} also holds. Similarly, when $w\equiv m$ in $\bar D\times [0,T]$, again by \eqref{4.6}, $m\geq 0$ and therefore, $m=0$ and
$w\equiv 0$ and \eqref{4.8} holds, i.e.,
\begin{equation*}
   u(x,t)\geq \e v(x) \quad \hbox{for all} \;\; (x,t)\in \bar D\times[0,T].
\end{equation*}
Furthermore, for every  $x \in D$,
\begin{align*}
v(x) & =e^{\a (R^2-|x-x_0|^2)}-1=e^{\a (R-|x-x_0|)(R+|x-x_0|)}-1\cr & \geq e^{\a \frac {3R}{2}(R-|x-x_0|)}-1
\geq \a \tfrac{3R}{2}(R-|x-x_0|)  =\a \tfrac{3R}{2}  \, \mathrm{dist\,}(x, \partial B_R(x_0)).
\end{align*}
Thus, as soon as $(x,t)\in D \times [0,T]$,  we find that
\begin{equation}
\label{4.9}
   u(x,t)\geq \e v(x) \geq \frac {3 \a R}{2(e^{\a R^2}-1)}u_R \, \mathrm{dist\,}(x,\partial B_R(x_0)).
\end{equation}
On the other hand, for every $x\in \bar B_{\frac{R}{2}}(x_0)$,
\begin{equation}
\label{4.10}
   u(x,t)\geq u_R \geq  \frac {u_R}{R} \,\mathrm{dist\,}(x, \partial B_R(x_0)).
\end{equation}
Therefore, setting
\begin{equation*}
M:=\min \left \{ \frac {3 \a R}{2(e^{\a R^2}-1)},\frac {1}{R} \right \},
\end{equation*}
from \eqref{4.9} and \eqref{4.10} it becomes apparent that
\begin{equation}
\label{4.11}
   u(x,t)\geq M u_R \,\mathrm{dist\,}(x,\partial B_R(x_0)) \quad \hbox{for all}\quad (x,t)\in \bar {B}_R(x_0)\times [0,T].
\end{equation}
We are ready to complete the proof of the theorem. By Assumption {\rm (A1)}, $\O$ satisfies the uniform interior sphere property in the strong sense on $\G_0$ with parameter $R>0$ (see \cite[Ch. 1]{LG13}, if necessary).  Subsequently, we will consider the compact subset of $\bar\O$ defined by
\begin{equation*}
   K_R=\left \{(x,t)\in \bar{Q}_T : \, \mathrm{dist\,}(x,\G_0)\geq \tfrac{R}{2} \right\}.
\end{equation*}
We already know that $u$ satisfies \eqref{3.18} and \eqref{3.19} (see the beginning of this proof). In particular, since $\mc{P}u\geq 0$, by \eqref{3.18} we have that
\begin{equation*}
 \left \{ \begin{array}{ll}u(x,t)> 0 &\quad   (x,t)\in (\bar{\O} \setminus \G_0) \times [0,T], \\
   \mathcal{P}u(x,t) \geq  0 &\quad (x,t)\in \O \times [0,T]. \end{array}\right.
\end{equation*}
Thus,
\begin{equation*}
 u_L:=\min_{K_R}u>0.
\end{equation*}
Let $(x,t)\in Q_T$  with $\textrm{dist}(x,\G_0)\leq R$  and consider $y_x \in \G_0$ such that
\begin{equation*}
\textrm{dist}(x,\G_0)=|x-y_x| \quad \hbox{and}  \quad B_R(x_0)\subset \O,
\end{equation*}
where
\begin{equation*}
x_0:=y_x+R \frac {x-y_x}{|x-y_x|}.
\end{equation*}
Since $\bar{B}_{\frac{R}{2}}(x_0) \times [0,T] \subset K_R$, we have that
\begin{equation*}
\min_{\bar{B}_{\frac{R}{2}}(x_0) \times [0,T]}u \geq  u_L.
\end{equation*}
Thus, according to \eqref{4.11}, there exists a constant $M=M(\mathfrak{L},R)$, independent of $x$, such that,
whenever $(x,t)\in Q_T$  with $\, \mathrm{dist\,}(x,\G_0)\leq R$,
\begin{equation*}
   u(x,t)\geq M u_L \, \textrm{dist}(x,\partial B_R(x_0))=Mu_L|x-y_x|=M u_L \textrm{dist}(x,\G_0).
\end{equation*}
Finally, let $(x,t)\in Q_T$ be with $\, \mathrm{dist\,}(x,\G_0)>R$. Obviously,
\begin{equation*}
   u(x,t)\geq \min_{(x,t)\in K_R}  \frac {u(x,t)}{\,  \mathrm{dist\,}(x,\G_0)} {\,  \mathrm{dist\,}}(x,\G_0).
\end{equation*}
Moreover,
\begin{equation*}
   \eta :=\min_{(x,t)\in K_R}  \frac {u(x,t)}{\, \mathrm{dist\,}(x,\G_0)}>0.
\end{equation*}
Thus,  setting
\begin{equation*}
   \d:= \min \left \{\eta ,Mu_L \right\},
\end{equation*}
the estimate \eqref{4.1} holds. This ends the proof.
\end{proof}

\setcounter{equation}{0}

\section{Proof of Theorem \ref{th11}}

This section gives the proof of Theorem \ref{th11}.  Thanks to Theorem  \ref{4.1}, the supersolution $h$ satisfies
\begin{equation}
\label{5.1}
   h(x,t)>0 \qquad \hbox{for all}\quad (x,t)\in \left(\O\cup\G_1\right)\times [0,T],
\end{equation}
and
\begin{equation}
\label{5.2}
  \frac{\p h}{\p \nu}(x,t)<0 \qquad \hbox{for all}\quad   (x,t) \in \left[ h^{-1}(0)\cap \G_0\right]\times [0,T].
\end{equation}
Moreover, there exists $\d>0$ such that
\begin{equation}
\label{5.3}
  h(x,t) \geq \d\, \mathrm{dist\,}(x,\G_0)\qquad\hbox{for all}\;\;  (x,t)\in  Q_T.
\end{equation}
Let $u\in E$ be an arbitrary supersolution of $(\mathcal{P},\mathfrak{B},Q_T)$. Then, $\mathcal{P}u\geq 0$ in $Q_T$ and $\mathfrak{B}u\geq 0$ on $\p\O\times [0,T]$. Suppose $u\geq 0$, $u\neq 0$. Then, thanks again to  Theorem  \ref{4.1}, $u$ satisfies \eqref{3.18} and \eqref{3.19}. Hence, Alternative (b) holds.  Consequently, in case $u\geq 0$ one of the first two alternatives occurs.
\par
To complete the proof of the theorem it remains to prove that Alternative (c) holds if $u$ is somewhere negative. So, suppose
\begin{equation}
\label{5.4}
  u(x_0,t_0)<0 \qquad \hbox{for some}\quad (x_0,t_0)\in \O\times [0,T].
\end{equation}
Subsequently, for every $\l\geq 0$, we consider the function
\begin{equation*}
  v_\l(x,t):= u(x,t)+\l h(x,t), \qquad (x,t)\in\bar Q_T.
\end{equation*}
Naturally, $v_\l \geq 0$ in $\bar Q_T$ for sufficiently large $\l>0$. To prove it, we will argue by contradiction. Assume that for each integer $k\geq 1$ there is
$(x_k,t_k)\in Q_T$ such that
\begin{equation}
\label{5.5}
  v_k(x_k,t_k) = u(x_k,t_k)+k h(x_k,t_k)<0.
\end{equation}
As $\bar Q_T$ is compact, there exist $x_\o\in\bar\O$, $t_\o\in [0,T]$, and a subsequence of $\{k\}_{k\geq 1}$, say $\{k_m\}_{m\geq 1}$, such that
\begin{equation*}
  \lim_{m\to\infty} \left(x_{k_m},t_{k_m}\right)=(x_\o,t_\o).
\end{equation*}
Thanks to \eqref{5.5}, we have that
\begin{equation}
\label{5.6}
  \frac{1}{k_m}u(x_{k_m},t_{k_m})+h(x_{k_m},t_{k_m})<0,\qquad m\geq 1.
\end{equation}
Moreover, by the continuity of $u$ in $\bar Q_T$, it is apparent that
\begin{equation*}
  \lim_{m\to\infty} \frac{1}{k_m} u(x_{k_m},t_{k_m})=0.
\end{equation*}
Thus, letting $m\to\infty$ in \eqref{5.6} yields $h(x_\o,t_\o) \leq 0$. Hence, by \eqref{5.1}, $h(x_\o,t_\o)=0$ with $x_\o \in \G_0$.
In particular, $\G_0\neq \emptyset$. Consequently, the proof of the previous assertion get completed if $\G_0=\emptyset$. Suppose
$\G_0\neq \emptyset$ and, for each $m\geq 1$, let $y_{k_m}\in \G_0$ such that
\begin{equation*}
  \mathrm{dist\,}(x_{k_m},\G_0)=|x_{k_m}-y_{k_m}|.
\end{equation*}
Since $u$ is a supersolution of $(\mathcal{P},\mathfrak{B},Q_T)$, we have that $u\geq 0$ on $\G_0$. Hence,
\begin{equation*}
  -u(x_{k_m},t_{k_m})\leq u(y_{k_m},t_{k_m})-u(x_{k_m},t_{k_m}),\qquad m\geq 1.
\end{equation*}
Thus, since $u\in E$, there exists a constant $L>0$ such that
\begin{equation}
\label{5.7}
  -u(x_{k_m},t_{k_m})\leq L   |y_{k_m}-x_{k_m}|=L \,\mathrm{dist\,}(x_{k_m},\G_0),\qquad   m\geq 1.
\end{equation}
So, combining \eqref{5.6} with \eqref{5.7} yields
\begin{equation*}
  k_m h(x_{k_m},t_{k_m})<L\,\mathrm{dist\,}(x_{k_m},\G_0),\qquad   m\geq 1.
\end{equation*}
Therefore, by \eqref{5.3}, it becomes apparent that
\begin{equation}
\label{5.8}
  \d k_m \mathrm{dist\,}(x_{k_m},\G_0)<L   \,\mathrm{dist\,}(x_{k_m},\G_0), \qquad m\geq 1.
\end{equation}
As \eqref{5.8} implies $\d k_m < L$, $m\geq 1$, which is impossible because $\lim_{m\to\infty}k_m=\infty$, the proof of the existence of a $\l>0$ for which $v_\l \geq 0$ is completed.
\par
Subsequently, we will denote by $\Lambda$  the set
\begin{equation*}
  \L := \{ \l > 0\;:\; v_\l\geq 0\}.
\end{equation*}
We have just shown that $\L\neq \emptyset$. Moreover, by \eqref{5.4}, $\l \not \in \L$ for sufficiently
small $\l\geq 0$, and, since $h\geq 0$, it becomes apparent that $[\l,\infty)\subset \L$ if $\l\in\L$. Therefore,
\begin{equation*}
  \mu:= \inf \L > 0.
\end{equation*}
It is clear that $v_\mu \geq 0$, by continuity. Moreover,
\begin{align*}
  \mathcal{P}v_\mu & = \mathcal{P}u+\mu \mathcal{P}h \geq 0
  \quad \hbox{in}\quad Q_T, \cr
    \mathfrak{B}v_\mu & = \mathfrak{B}u+\mu \mathfrak{B}h \geq 0
  \quad \hbox{on}\quad \p\O\times [0,T].
\end{align*}
Thus, $v_\mu$ is a non-negative supersolution of $(\mathcal{P},\mathfrak{B},Q_T)$ in $E$. According to Corollary \ref{co31}, either
\begin{equation}
\label{5.9}
  v_\mu = 0,
\end{equation}
or
\begin{equation}
\label{5.10}
\left\{ \begin{split}
  & v_\mu(x,t)>0 \qquad \forall\, (x,t)\in \left( \O\cup\G_1\right)\times [0,T], \cr
  & \p_\nu v_\mu (x,t)<0 \quad
  \forall \,(x,t) \in \left[ v_{\mu}^{-1}(0)\cap \G_0\right]\times [0,T].
\end{split}\right.
\end{equation}
Moreover, thanks to Theorem \ref{th41}, when \eqref{5.10} occurs, there exists $\d=\d(\mu)$ such that
\[
  v_\mu(x,t)\geq \d(\mu)\mathrm{dist\,}(x,\G_0)\quad \hbox{for all}\;\; (x,t)\in \bar Q_T.
\]
Suppose \eqref{5.9} holds. Then, setting $m:=-\mu<0$ yields $u=m h$ and Alternative (c) holds; \eqref{1.7} follows very easily using the fact that $u$ and $h$ are supersolutions of
$(\mathcal{P},\mathfrak{B},Q_T)$. So, to complete the proof of the theorem it suffices to show that \eqref{5.10} contradicts the minimality of $\mu$. Indeed, by definition of $\mu$, for each $k\geq 1$ there exists
$(x_k,t_k)\in Q_T$ such that
\begin{equation}
\label{5.11}
  v_{\mu-\frac{1}{k}}(x_k,t_k)=u(x_k,t_k)+\left( \mu-\frac{1}{k}\right)
  h(x_k,t_k)=v_\mu(x_k,t_k)-\frac{h(x_k,t_k)}{k} <0.
\end{equation}
Arguing as above, there exist $(x_\o,t_\o)\in\bar Q_T$ and a subsequence of
$\{k\}_{k\geq 1}$, $\{k_m\}_{m\geq 1}$, such that
\begin{equation*}
  \lim_{m\to\infty} \left(x_{k_m},t_{k_m}\right)=(x_\o,t_\o).
\end{equation*}
Thanks to \eqref{5.11},
\begin{equation}
\label{5.12}
  v_\mu(x_{k_m},t_{k_m})<\frac{h(x_{k_m},t_{k_m})}{k_m},\qquad m\geq 1.
\end{equation}
On the other hand, by the continuity of $h$ in $\bar Q_T$,
\begin{equation*}
  \lim_{m\to\infty}\frac{h(x_{k_m},t_{k_m})}{k_m}=0.
\end{equation*}
Thus, letting $m\to\infty$ in \eqref{5.12} shows that $v_\mu(x_\o,t_\o)\leq 0$. Therefore,
$v_\mu(x_\o,t_\o)= 0$ and hence, due to \eqref{5.10}, we find that
\begin{equation*}
 (x_\o,t_\o) \in \left[ v_\mu^{-1}(0)\cap \G_0\right] \times [0,T].
\end{equation*}
As above,  this entails $\G_0\neq \emptyset$ and ends the proof if $\G_0=\emptyset$. So, suppose $\G_0\neq \emptyset$, and, for each $m\geq 1$, let $y_{k_m}\in \G_0$ such that
\begin{equation*}
 \mathrm{dist\,}(x_{k_m},\G_0)=|x_{k_m}-y_{k_m}|.
\end{equation*}
Arguing as above shows that \eqref{5.8} holds. On the other hand, by \eqref{5.11}, we have that
\begin{equation*}
  -u(x_{k_m},t_{k_m}) > \left( \mu-\frac{1}{k_m}\right) h(x_{k_m},t_{k_m}),\qquad m\geq 1.
\end{equation*}
So, according to \eqref{5.7}, we find that
\begin{equation*}
  \left( \mu-\frac{1}{k_m}\right)    h(x_{k_m},t_{k_m})< L \,\mathrm{dist\,}(x_{k_m},\G_0),
  \qquad m\geq 1.
\end{equation*}
Thus,  since $\lim_{m\to\infty}k_m=\infty$, for sufficiently large $m\geq 1$ we obtain that
\begin{equation*}
   h(x_{k_m},t_{k_m}) < \frac{L k_m}{\mu k_m-1}\,\mathrm{dist\,}(x_{k_m},\G_0).
\end{equation*}
Consequently, going back to \eqref{5.11}, we find that for sufficiently large $m\geq 1$,
\begin{align*}
  0 & > v_\mu(x_{k_m},t_{k_m})-\frac{h(x_{k_m},t_{k_m})}{k_m} \cr & >
  \d(\mu) \, \mathrm{dist\,}(x_{k_m},\G_0)-\frac{L }{\mu
   k_m-1}\,\mathrm{dist\,}(x_{k_m},\G_0) \\ & =
   \left( \d(\mu) - \frac{L }{\mu   k_m-1}\right) \,\mathrm{dist\,}(x_{k_m},\G_0)
\end{align*}
because $v_\mu$ is a positive supersolution of $(\mathcal{P},\mathfrak{B},Q_T)$ satisfying \eqref{5.10}.
As the previous inequality cannot be satisfied, because it entails $\mathrm{dist\,}(x_{k_m},\G_0)<0$ for sufficiently large $m\geq 1$, the proof is completed.

\section{The maximum principle for cooperative systems}

\noindent  Throughout this section,  for every $k\in \{1,...,M\}$, we will suppose that $\p \O$ consists of two disjoint open and closed subsets, $\G_{0,k}$ and $\G_{1,k}$,
\begin{equation*}
  \p \O :=\G_{0,k}\cup \G_{1,k}.
\end{equation*}
As they are disjoint, $\G_{0,k}$ and $\G_{1,k}$ must be of class  $\mc{C}^{2+\t}$ for all $k\in \{1,...,M\}$. Then, for every $k\in\{1,...,M\}$, we consider the boundary operators
\begin{equation*}
    \mathfrak{B}_k: \mathcal{C}(\G_{0,k})\oplus\mathcal{C}^1(\O\cup \G_{1,k}) \to \mc{C}(\p\O)
\end{equation*}
defined by
\begin{equation}
\label{vi.1}
    \mathfrak{B}_k \xi := \left\{ \begin{array}{ll}
    \xi \qquad & \hbox{on } \;\;\G_{0,k} \\     \frac{\p \xi}{\p \nu} + \b_k(x) \xi \qquad &
   \hbox{on } \;\; \G_{1,k}  \end{array} \right.
\end{equation}
for each $\xi\in \mathcal{C}(\G_{0,k})\oplus \mathcal{C}^1(\O\cup\G_{1,k})$, where the coefficients $\b_k\in \mathcal{C}^{1+\t}(\G_{1,k})$, $1\leq k\leq M$, can change of sign,  and $\nu =(\nu_1,...,\nu_N)\in \mathcal{C}^{1+\t}(\p\O;\R^N)$ is an outward pointing nowhere tangent vector field. When $\G_{1,k}=\emptyset$
for some $k \in \{1,...,M\}$, $\mf{B}_k$ will be simply denoted by $\mf{D}_k$, the Dirichlet operator. Eventually, we will set
\[
  \mc{B}:=\left(\mf{B}_1,...,\mf{B}_M\right), \quad \mc{D}:=\left(\mf{D}_1,...,\mf{D}_M\right).
\]
For a given $T>0$, we also consider the next non-autonomous differential operators
\begin{equation}
\label{vi.2}
  \mc{A}_k:=\mc{A}_k(x,t;D):= -\sum_{i,j=1}^N a_{ij}^k(x,t)\frac{\p^2}
  {\p x_i \p x_j}+\sum_{j=1}^N b_j^k(x,t) \frac{\p}{\p x_j},
\end{equation}
where $a_{ij}^k=a_{ji}^k, b_j^k \in F$ for all $i, j \in\{1,...,N\}$ and $k \in\{1,...,M\}$, as well as the associated parabolic operators
\begin{equation}
\label{vi.3}
    \mc{P}_k:=\p_t +\mc{A}_k(x,t;D),\qquad    1 \leq k\leq M.
\end{equation}
Finally, we  consider the cooperative matrix of order $M$,
\begin{equation*}
  \mathcal{C}(x,t)=\left(c_{ij}(x,t)\right)_{i,j=1,...,M}
\end{equation*}
 where $c_{ij}\in F$ for every $i, j \in \{1,...,M\}$; by cooperative it is meant that
\begin{equation}
\label{vi.4}
  c_{ij}(x,t)>0 \quad \hbox{for every} \; \; i, j \in\{1,...,M\}, \;\; i\neq j, \quad (x,t)\in\bar Q_T.
\end{equation}
The main goal of this section is generalizing Theorem 2.2 of \cite{ALG} to cover the more general periodic parabolic problem
\begin{equation}
\label{vi.5}
\left\{ \begin{array}{ll}
    \mc{P}_k u_k=\displaystyle{\sum_{j=1}^M\ c_{kj} u_j+f_k} & \quad\hbox{in}\quad Q_T\equiv\O \times[0,T], \\
    \mathfrak{B}_ku_k = 0 & \quad\hbox{on}\quad \p\O\times[0,T], \\
 \end{array}\right. \quad 1\leq k \leq M,
\end{equation}
where $f_k\in F$, $1\leq k \leq M$,   are given functions. In particular, we will characterize whether or not \eqref{vi.5}
satisfies the strong maximum principle. Note that Theorem 2.2 of \cite{ALG} was established in the very special case when $\G_{1,k}=\emptyset$ for all $k\in\{1,...,M\}$.
\par
As in \cite{ALG}, we consider the Banach spaces
\[
  U=E_\mc{B}:=E_{\mf{B}_1}\times \cdots\times E_{\mf{B}_M}, \qquad V:=F^M,
\]
as well as the operator
\begin{equation*}
  \mc{P}:=\left(\mc{P}_1,...,\mc{P}_M\right):U\rightarrow V
\end{equation*}
defined by
\begin{equation*}
  \mc{P}  u =\left( \mc{P}_1 u_1,...,\mc{P}_M u_M\right)^\mathrm{T},\quad u=(u_1,...,u_M)^\mathrm{T} \in U,
\end{equation*}
where $T$ stands for \lq\lq transposition\rq\rq, so that \eqref{vi.5}  can be expressed in the compact way
\begin{equation}
\label{vi.6}
   \mc{P} u = \mc{C}(x,t) u+ f(x,t) ,\quad u \in U,
\end{equation}
where
\begin{equation*}
   f =(f_1,...,f_M)^\mathrm{T} \in V.
\end{equation*}
Let us denote by $P_E$ and $P_F$  the cones of positive functions of $E$ and $F$, respectively.
Subsequently, the  Banach spaces $E$, $F$, $U$ and  $V$ are viewed as ordered Banach spaces with cones   $P_E$, $P_F$,
\[
  P_U:=P_{E_{\mf{B}_1}}\times \cdots \times P_{E_{\mf{B}_M}}\quad \hbox{and}\quad
  P_V:=P_F\times \cdots \times P_F,
\]
respectively. Given an arbitrary ordered Banach space, $X$, with cone $P_X$, $\mathrm{int\,} P_X$ stands for the interior of $P_X$. Moreover, given $x, y \in X$ it is said that $x>y$ if $x-y \in P_X\setminus\{0\}$, while we will write $x\gg y$ if
$x-y \in \mathrm{int\,}P_X$. It is folklore that, for every $k\in \{1,...,M\}$,
$\mathrm{int\,}P_{E_{\mathfrak{B}_k}}$ consists of the set of functions $u\in E_{\mathfrak{B}_k}$ such that
$u(x,t)>0$ for all $(x,t)\in (\O\cup \G_{1,k})\times [0,T]$ and $\frac{\p u}{\p \nu}(x,t)<0$ for all $(x,t)\in \G_{0,k} \times [0,T]$.
Note that $u=0$ on $\G_{0,k}\times [0,T]$ for all $u \in E_{\mf{B}_k}$. Moreover,
\begin{equation*}
  \mathrm{int\,} P_U=  \mathrm{int\,} P_{E_{\mathfrak{B}_1}}\times \cdots \times  \mathrm{int\,} P_{E_{\mathfrak{B}_M}}.
\end{equation*}
The main result of this section is the next one.

\begin{theorem}
\label{th6.1}
 The following conditions are equivalent:
\begin{itemize}
\item[{\rm (C1)}] There exits $\Psi \in \mathrm{int\, } P_U$ such that $\mc{P} \Psi >\mc{C}\Psi$.
\item[{\rm (C2)}] The operator $(\mc{P}-\mc{C})^{-1} :V\rightarrow V$ is well defined, compact and strongly positive, in the sense that $f \in V$, $f>0$, implies $(\mc{P}-\mc{C})^{-1}f \gg 0$.
\item[{\rm (C3)}] The problem \eqref{vi.6} satisfies the strong maximum principle, in the sense that
\[
  f=(f_1,...,f_M)^T>0\;\; \hbox{implies}\;\; u \gg 0\;\; \hbox{in}\;\; U.
\]
\item[{\rm (C4)}] The problem \eqref{vi.6} satisfies the maximum principle, in the sense that
\[
  f=(f_1,...,f_M)^T\geq 0\;\; \hbox{implies}\;\; u \geq 0\;\; \hbox{in}\;\; U,
\]
for all solution, $u$, of \eqref{vi.6}.
\item[{\rm (C5)}] The eigenvalue problem
\begin{equation}
\label{vi.7}
   \mc{P} \phi = \mc{C}(x,t) \phi+ \sigma \phi ,\quad \phi \in U,
\end{equation}
  admits a positive  eigenvalue, $\s [\mc{P}-\mc{C}]$,  associated to a positive eigenfunction, $\Phi \in \mathrm{int\,} P_U$, which is unique up to a multiplicative constant.
\end{itemize}
Moreover, if one of these conditions holds, the eigenvalue $\s [\mc{P}-\mc{C}]$ is simple and there is no any other eigenvalue of  \eqref{vi.7} associated to a positive eigenfunction.
\par
Finally, for each $p \in P_U\setminus\{0\}$, the equation
\begin{equation}
\label{6.8}
\lambda u-(\mc{P}-\mc{C})^{-1} u= p,
\end{equation}
has a unique positive solution, $ u\in \mathrm{int\,}P_U$, if
\[
  \lambda > \mathrm{spr\,}(\mc{P}-\mc{C})^{-1}= \frac{1}{\s [\mc {P}-\mc{C}]},
\]
and no positive solution if
\[
  \lambda \leq \mathrm{spr\,}(\mc{P}-\mc{C})^{-1}=\frac {1}{\s [\mc {P}-\mc{C}]}.
\]
\end{theorem}

\begin{proof}

The proof of this theorem is based on the next generalized version of the Krein--Rutman theorem, \cite{KR}.

\begin{theorem}
\label{thKR}
 Let $(X,\|\cdot\|,P)$ be an ordered  Banach space with $\mathrm{int\,}P\neq \emptyset$ and $T \in \mc{K}(X)$ a compact operator
 such that it is strongly positive in the sense that
\[
  T(P\setminus\{0\})\subset \mathrm{int\,}P.
\]
Then,
\begin{enumerate}
\item[{\rm (a)}] $\mathrm{spr\,}T>0$ is an algebraically simple eigenvalue of $T$ with
\[
  N[\mathrm{spr\,}T I-T] =\mathrm{span\,}[x_0]
\]
for some $x_0\in \mathrm{int\,}P$.
\item[{\rm (b)}] $\mathrm{spr\,}T$ is the unique real eigenvalue to an eigenvector in $P\setminus\{0\}$.
\item[{\rm (c)}] $\mathrm{spr\,}T$ is the unique eigenvalue of $T$ in the spectral circle
\[
  |\zeta|=\mathrm{spr\,}T.
\]
In other words,
\[
  |\l|<\mathrm{spr\,}T \;\; \hbox{for all}\;\; \l \in \sigma(T)\setminus \{\mathrm{spr\,}T\}.
\]
\item[{\rm (d)}] For every real number $\l > \mathrm{spr\,}T$, the resolvent operator
\[
  \mc{R}(\l,T):= (\l I-T)^{-1}\in \mc{L}(X)
\]
is strongly positive, i.e.,
\[
  \mc{R}(\l,T)(P\setminus\{0\})\subset \mathrm{int\,} P.
\]
\item[{\rm (e)}] Conversely, for every $p \in P\setminus\{0\}$, the equation
\[
  \l x - T x = p
\]
cannot admit a positive solution for $\l \leq \mathrm{spr\,} T$.
\end{enumerate}
\end{theorem}
\vspace{0.2cm}

See \cite[Th. 6.3]{LG13} for a
complete statement and the proof of Parts (a)-(d), as well as \cite[Th. 3.2(iv)]{Amann} for Part (e).
The original theorem of  Krein and  Rutman \cite{KR}  reads as follows.

\begin{theorem}
Let $X$ be an ordered Banach space with total positive cone $P$, and let $T$ be a compact positive endomorphism of $X$. If $T$
has a spectral radius $\mathrm{spr\,}T>0$, then $\mathrm{spr\,}T$ is a pole of the resolvent of maximal order on the spectral circle
\[
 |\zeta|=\mathrm{spr\,}T,
\]
with an eigenvector in $P$. A corresponding result holds for the adjoint $T^*$ of $E'$.
\end{theorem}

Thus, Theorem \ref{thKR} is substantially sharper. The interested readers are sent to \cite[Ch. 6]{LG13} for a more detailed discussion.
\vspace{0.2cm}

We are ready to prove Theorem \ref{th6.1}.
First, we will show ${\rm (C1)} \Rightarrow {\rm (C2)}$. Let
\begin{equation*}
   \Psi =(\psi_1,...,\psi_M)\in \mathrm{int\,}\mc {P}_U= \mathrm{int\,} P_{E_{\mathfrak{B}_1}}\times \cdots \times  \mathrm{int\,} P_{E_{\mathfrak{B}_M}}
\end{equation*}
be such that
\begin{equation*}
  \mc{P} \Psi >\mc{C} \Psi.
\end{equation*}
Then, $\psi_i \in E_{\mf{B}_i}$ satisfies $\psi_i\gg 0$, i.e., $\psi_i\in \mathrm{int\,} P_{E_{\mathfrak{B}_i}}$, and, in particular,
$\psi_i \in X_i$, where $X_i$ stands for the Banach space of all $T$-periodic functions in
$\mathcal{C}^{1,\frac{1}{2}}(\bar\O\times \R;\R)$ such that $\mf{B}_i \psi_i =0$ on $\p\O\times \R$. Note that
$\psi_i\gg 0$ if and only if $\psi_i(x,t)>0$ for all $x\in\O\cup \G_{1,i}$ and $t\in\R$, and $\p_\nu \psi_i(x,t)<0$ for all $x\in\G_{0,i}$ and $t\in\R$. These Banach spaces can be ordered by their respective cones of non-negative functions
\[
  P_{X_i}:=\{u\in X_i\;:\;\; u\geq 0 \;\;\hbox{in}\;\; \bar Q_T\},\qquad 1\leq i \leq M.
\]
Moreover, $E_{\mf{B}_i}$ is compactly embedded in $X_i$ for all $1\leq i \leq M$, which provides us with
the compactness of most of the operators arising in this proof.
Since $\psi_i\gg 0$ in $\bar Q_T$,
\begin{equation*}
 \mc{P}_i \psi_i \geq \sum_{j=1}^M\ c_{ij} \psi_j > c_{ii} \psi_i \qquad  \hbox{for every}\;\;\; i\in\{1,...,M\}.
\end{equation*}
Thus, $\psi_i$ provides us with a positive strict
supersolution of $(\mc{P}_i - c_{ii},\mathfrak{B}_i,Q_T)$ and, thanks to Theorem \ref{th12},
\[
   \l_1[\mc{P}_i - c_{ii},\mathfrak{B}_i,Q_T]>0,\qquad 1\leq i \leq M.
\]
Moreover, by parabolic regularity, for every $i\in\{1,...,M\}$,
\begin{equation*}
 (\mc{P}_i - c_{ii} )^{-1} : F \rightarrow X_i
\end{equation*}
is compact and strongly order preserving, because
\[
  (\mc{P}_i - c_{ii} )^{-1} (F)\subset E_{\mf{B}_i},
\]
the injection $E_{\mf{B}_i}\hookrightarrow X_i$ is compact, and $(\mc{P}_i - c_{ii},\mathfrak{B}_i,Q_T)$ satisfies the strong
maximum principle.
\par
Subsequently, we consider the vectorial function
\begin{equation*}
   f :=  \mc{P}\Psi - \mc {C} \Psi.
\end{equation*}
By hypothesis, $f >0$. And the $i$-th equation of the previous identity establishes that
\[
  f_i = (\mc{P}_i-c_{ii})\psi_i- \sum_{\substack{j=1\cr j\neq i}}^M c_{ij}\psi_j, \qquad 1 \leq i \leq M.
\]
Consequently, acting $(\mc{P}_i - c_{ii} )^{-1}$ on this identity yields
\begin{equation}
\label{6.9}
\psi_i=\sum_{\substack{j=1\cr j\neq i}}^M (\mc{P}_i-\ c_{ii})^{-1}(c_{ij}\psi_j)+(\mc{P}_i-\ c_{ii})^{-1}f_i, \qquad
1\leq i \leq M.
\end{equation}
Since $f>0$,  without loss of generality, we can assume that
\begin{equation*}
  f_M>0.
\end{equation*}
It suffices to reorder and relabel each of the components, if necessary.
\par
We begin by proving (C2) in the special case when $M=2$. In this case, the second equation of \eqref{6.9} becomes
\begin{equation}
\label{6.x}
    \psi_2=(\mc{P}_2-\ c_{22})^{-1}(c_{21}\psi_1)+(\mc{P}_2-\ c_{22})^{-1}f_2.
\end{equation}
Hence, substituting this identity into the first equation of  \eqref{6.9} yields
\begin{equation}
\label{6.xi}
\begin{split}
  \psi_1 & = (\mc{P}_{1}-c_{11})^{-1}\left( c_{12} \psi_2\right) + (\mc{P}_{1}-c_{11})^{-1} f_1 \cr
         & = (\mc{P}_{1}-c_{11})^{-1}\left( c_{12} (\mc{P}_2-\ c_{22})^{-1}(c_{21}\psi_1) \right) \cr & \qquad\;\;\; +
         (\mc{P}_{1}-c_{11})^{-1}\left( c_{12} (\mc{P}_2-\ c_{22})^{-1} f_2 \right)+ (\mc{P}_{1}-c_{11})^{-1} f_1.
\end{split}
\end{equation}
Thus, setting
\begin{align*}
  \mc{T}_{11}^1 & = (\mc{P}_{1}-c_{11})^{-1} ( c_{12} (\mc{P}_2-\ c_{22})^{-1}(c_{21}\, \cdot )) \cr
  \mc{K}_{11}^1 & = (\mc{P}_{1}-c_{11})^{-1} \cr
  \mc{K}_{12}^1 & = (\mc{P}_{1}-c_{11})^{-1}\left( c_{12} (\mc{P}_2-\ c_{22})^{-1} \, \cdot \right)
\end{align*}
the identity \eqref{6.xi} can be equivalently expressed as
\begin{equation}
\label{6.xii}
  \psi_1 =  \mc{T}_{11}^1 \psi_1 + \mc{K}_{11}^1 f_1 + \mc{K}_{12}^1 f_2.
\end{equation}
Since the operators $\mc{T}_{11}^1, \mc{K}_{11}^1, \mc{K}_{12}^1 : F \to X_1$  are compact and strongly positive, and $f_1\geq 0$, $f_2>0$, we find that
\[
  \psi_1 > \mc{T}_{11}^1 \psi_1
\]
and hence, thanks to Theorem \ref{thKR}(e),
\begin{equation}
\label{6.xiii}
  \mathrm{spr\,} \mc{T}_{11}^1 < 1
\end{equation}
because $\psi_1>0$. To infer \eqref{6.xiii} we are applying Theorem \ref{thKR}(e) to the operator
$\mc{T}_{11}^1$ in the Banach space $X_1$ ordered by $P_{X_1}$. According to \eqref{6.xiii}, by Theorem \ref{thKR}(d), the compact operator
\[
  (I_{X_1}-\mc{T}_{11}^1)^{-1}:X_1\to X_1
\]
is strongly positive and \eqref{6.xii} provides us with
\begin{equation}
\label{6.xiv}
  \psi_1= \left(I_{X_1}-\mc{T}_{11}^1\right)^{-1}\mc{K}_{11}^1 f_1+ \left(I_{X_1}-\mc{T}_{11}^1\right)^{-1}\mc{K}_{12}^1 f_2.
\end{equation}
Note that  $\psi_1$ is  uniquely determined by $f=(f_1,f_2)$. Consequently, setting
\[
  \mc{R}_{11}:=\left(I_{X_1}-\mc{T}_{11}^1\right)^{-1}\mc{K}_{11}^1,\qquad
  \mc{R}_{12} := \left(I_{X_1}-\mc{T}_{11}^1\right)^{-1}\mc{K}_{12}^1,
\]
the identity \eqref{6.xiv} can be expressed as
\begin{equation}
\label{6.xv}
  \psi_1= \mc{R}_{11} f_1+ \mc{R}_{12} f_2.
\end{equation}
Moreover, the resolvent operators $\mc{R}_{11}$ and $\mc{R}_{12}$, viewed as operators from $F$ to  $X_1$, are compact and strongly positive (with images in $E_{\mf{B}_1} \hookrightarrow X_1$).
\par
Similarly, substituting \eqref{6.xv} into \eqref{6.x} yields
\begin{align*}
 \psi_2 & =(\mc{P}_2-\ c_{22})^{-1}(c_{21}\psi_1)+(\mc{P}_2-\ c_{22})^{-1}f_2\cr
   &  = (\mc{P}_2-\ c_{22})^{-1}(c_{21}( \mc{R}_{11} f_1+ \mc{R}_{12} f_2))+(\mc{P}_2-\ c_{22})^{-1}f_2\cr
   & = \mc{R}_{21}f_1+\mc{R}_{22}f_2,
\end{align*}
where we have denoted
\begin{align*}
  \mc{R}_{21} & :=  (\mc{P}_2-\ c_{22})^{-1}(c_{21}( \mc{R}_{11} \, \cdot )), \cr
  \mc{R}_{22} & := (\mc{P}_2-\ c_{22})^{-1}(c_{21}( \mc{R}_{12} \, \cdot ))+(\mc{P}_2-\ c_{22})^{-1}.
\end{align*}
Reiterating the previous scheme, it becomes apparent that
\begin{equation}
\label{18J}
\begin{split}
  u_1 & = \mc{R}_{11} f_1+ \mc{R}_{12} f_2,\cr
  u_2 & = \mc{R}_{21} f_1+ \mc{R}_{22} f_2,
\end{split}
\end{equation}
provides us with the unique solution of  $\mc{P}u = \mc{C}u + f$ in $U$ for every $f\in V$. Consequently, (C2) holds in case $M=2$.
\par
Now, we assume that $M>2$. We will show that, for every $i \in \{1,...,M\}$, there are  $M$ compact and strongly positive operators
\begin{equation*}
 \mc{R}_{ij}: F\rightarrow X_i, \;\;\;\; j\in \{1,...,M\},
\end{equation*}
with images in $E_{\mc{B}_i}$, such that
\begin{equation*}
 \psi_i=\sum_{j=1}^M \mc{R}_{ij}f_j.
\end{equation*}
Substituting the last equation of  \eqref{6.9} into the previous ones and rearranging terms, we obtain that
\begin{equation}
\label{6.16}
 \psi_i=\sum_{j=1}^{M-1} \mc{T}_{ij}^1\psi_j+\mc{K}_{ii}^1f_i+\mc{K}_{iM}^1f_M, \qquad 1\leq i\leq M-1,
\end{equation}
for some compact and strongly positive operators
\begin{equation*}
 \mc{T}_{ij}^1, \; \mc{K}_{ii}^1,\; \mc{K}_{iM}^1 : F\rightarrow X_i,\qquad i, j \in \{1,...,M-1\},
\end{equation*}
constructed from $\mc{P}$ and $\mc{C}$, whose explicit expressions are irrelevant in the proof of Theorem \ref{th6.1}.
Note that \eqref{6.16} can be rewritten as
\begin{equation}
\label{6.17}
 \psi_i=\mc{T}_{ii}^1\psi_i+\sum_{\substack{j=1\cr j\neq i}}^{M-1} \mc{T}_{ij}^1\psi_j+\mc{K}_{ii}^1f_i+\mc{K}_{iM}^1f_M,
 \qquad 1\leq i \leq M-1.
\end{equation}
As $f_i \geq 0$, $\psi_i \gg 0$  for each $1\leq i \leq M$, and  $f_M>0$, \eqref{6.17} can be written down as
\begin{equation*}
 \psi_i=\mc{T}_{ii}^1\psi_i+p_i^1
\end{equation*}
for some $p_i^1\in \mathrm{int\, } X_i$, $i\in \{1,...,M-1\}$. Thus, since
$\psi_i \in \mathrm{int\, } X_i$, thanks to Theorem \ref{thKR}(e) applied in $X_i$,
\begin{equation*}
  \mathrm{spr\,}\mc{T}_{ii}^1 <1, \qquad 1 \leq i \leq M-1.
\end{equation*}
Hence, by Theorem \ref{thKR}(d), the compact operators
\begin{equation*}
   ( I_{X_i} - \mc{T}_{ii}^1)^{-1}:X_i\rightarrow X_i,  \qquad i \in \{1,...,M-1\},
\end{equation*}
are strongly positive, where $ I_{X_i} $ stands for the identity of $X_i$. Therefore, we deduce from \eqref{6.17} that
\begin{equation}
\label{6.18}
  \psi_i=( I_{X_i}\! - \!\mc{T}_{ii}^1)^{-1}\left(\sum_{\substack{j=1\cr j\neq i}}^{M-1} \mc{T}_{ij}^1\psi_j\!+\!\mc{K}_{ii}^1f_i\!+\!\mc{K}_{iM}^1f_M\right),
  \;\; 1\leq i \leq M-1.
\end{equation}
Fix $r\in \{1,2,...,M-1\}$. We will show that there are  $M$ compact and strongly positive operators
\begin{equation*}
 \mc{R}_{rj}: F\rightarrow X_r, \;\;\;\; j\in \{1,...,M\},
\end{equation*}
such that
\begin{equation}
 \label{6.19}
  \psi_r=\sum_{j=1}^M \mc{R}_{rj}f_j.
\end{equation}
Indeed, let $r_1 \in\{1,...,M-1\}\setminus \{r\}$ be and set
\begin{equation*}
  I_1:=\{1,...,M-1\}\setminus \{r_1\}.
\end{equation*}
Then, substituting  the $r_1$-equation  of \eqref{6.18} in the remaining ones and  rearranging terms yields
\begin{equation}
\label{6.20}
 \psi_i=\sum_{j\in I_1} \mc{T}_{ij}^2\psi_j+\mc{K}_{ii}^2f_i+\mc{K}_{ir_1}^2f_{r_1}+\mc{K}_{iM}^2f_M, \qquad  i\in I_1,
\end{equation}
 for some compact strongly positive operators,
 \begin{equation*}
\mc{T}_{ij}^2, \;  \mc{K}_{ii}^2, \;\mc{K}_{ir_1}^2,\; \mc{K}_{iM}^2  : F\rightarrow X_i,\qquad
 i, j \in I_1,
\end{equation*}
determined from $\mc{P}$ and $\mc{C}$. Arguing as before, \eqref{6.20} can be expressed as
\begin{equation*}
   \psi_i=\mc{T}_{ii}^2\psi_i+p_i^2, \quad i \in I_1,
\end{equation*}
for some $p_i^2\in \mathrm{int\,} X_i$, $i\in I_1$. Since
$\psi_i \in \mathrm{int\, } X_i$, by Theorem \ref{thKR}(e),
\begin{equation*}
  \mathrm{spr\,}\mc{T}_{ii}^2 <1, \qquad i \in I_1,
\end{equation*}
and hence, owing to Theorem \ref{thKR}(d), the compact operators
\begin{equation*}
( \mc{I}_{X_i} - \mc{T}_{ii}^2)^{-1}:\; X_i\rightarrow X_i,  \qquad  i \in I_1,
\end{equation*}
are strongly positive. Thus, from \eqref{6.20} we infer that, for every $i \in I_1$,
\begin{equation}
\label{6.21}
 \psi_i=\left( \mc{I}_{X_i} - \mc{T}_{ii}^2\right)^{-1} \left(\sum_{\substack{j=1\cr j\neq i, r_1}}^{M-1} \mc{T}_{ij}^2\psi_j+\mc{K}_{ii}^2f_i+\mc{K}_{ir_1}^2f_{r_1}+\mc{K}_{iM}^2f_M\right).
\end{equation}
If $I_1$ has an unique element, the proof is finished. Suppose that it has at least two, pick  $r_2 \in I_1\setminus \{r\}$ and set
\[
  I_2 := I_1\setminus \{r_2\}.
\]
Substituting the $r_2$-equation of \eqref{6.21} in the remaining ones and rearranging terms yields
\begin{equation}
\label{6.22}
 \psi_i=\sum_{j\in I_2} \mc{T}_{ij}^3\psi_j+\mc{K}_{ii}^3f_i+\mc{K}_{ir_1}^3f_{r_1}+\mc{K}_{ir_2}^3f_{r_2}+\mc{K}_{iM}^3f_M, \qquad
i\in I_2,
\end{equation}
for some compact strongly positive operators
 \begin{equation*}
 \mc{T}_{ij}^3, \; \mc{K}_{ii}^3, \;\mc{K}_{ir_1}^3, \;\mc{K}_{ir_2}^3,\; \mc{K}_{iM}^3 :\; F\rightarrow X_i,\qquad
 i, j \in I_2,
\end{equation*}
given by  $\mc{P}$ and $\mc{C}$, whose expressions are irrelevant in this proof. This recursive argument drives us to \eqref{6.19} in a finite number of steps. Indeed, in the last step we will be driven to the identity
\begin{equation}
\label{6.23}
 \psi_r=\mc{T}_{rr}^M \psi_r +\sum_{j=1}^{M-2} \mc{K}_{rr_j}^M f_j+ \mc{K}_{rr}^M f_r+ \mc{K}_{rM}^M f_M,
\end{equation}
for some compact strongly positive operators
 \begin{equation*}
\mc{T}_{rr}^M, \;  \mc{K}_{rj}^M\; :\;\; F\rightarrow X_r,\qquad 1\leq j \leq M,
\end{equation*}
constructed from $\mc{P}$ and $\mc{C}$. The equation \eqref{6.23} can be written down as
 \begin{equation*}
\psi_r=\mc{T}_{rr}^M\psi_r+p_r^M
\end{equation*}
for some  $p_r^M\in \mathrm{int\, } X_r$. As before, since
$\psi_r \in \mathrm{int\, } X_r$, Theorem \ref{thKR}(e) implies that
\begin{equation*}
  \mathrm{spr\,}\mc{T}_{rr}^M <1,
\end{equation*}
and hence, due to Theorem \ref{thKR}(d), the compact operator
\begin{equation*}
( \mc{I}_{X_r} - \mc{T}_{rr}^M)^{-1}:\; X_r\rightarrow X_r,
\end{equation*}
is strongly positive. Therefore, \eqref{6.23} becomes
\begin{equation*}
 \psi_r=\sum_{j=1}^{M-2}( I_{X_r} - \mc{T}_{rr}^M)^{-1} \mc{K}_{rr_j}^M f_j+( I_{X_r} - \mc{T}_{rr}^M)^{-1} \mc{K}_{rr}^M f_r+
 ( I_{X_r} -\mc{T}_{rr}^M)^{-1} \mc{K}_{rM}^M f_M
\end{equation*}
and setting
\[
  \mc{R}_{rj}:=( I_{X_r} - \mc{T}_{rr}^M)^{-1} \mc{K}_{rj}^M,\qquad j\in\{1,...,M\},
\]
we obtain that
\begin{equation*}
  \psi_r=\sum_{j=1}^M \mc{R}_{rj}f_j, \qquad  r\in \{1,2,...,M-1\}.
\end{equation*}
Finally, substituting these identities into the last equation of \eqref{6.9}, it is easily seen  that \eqref{6.16}
also holds for $i=M$. Summarizing, there are $M^2$ compact and  strongly positively operators,  $\mc{R}_{ij}:F\rightarrow X_i$,
$i, j \in\{1,...,M\}$,  such that
\begin{equation}
\label{6.24}
 \psi_i=\sum_{j=1}^M \mc{R}_{ij}f_j, \qquad i\in \{1,...,M\}.
\end{equation}
Since all the operators involved in this process are well defined, the last argument, also shows that
\begin{equation*}
   u: =\mc{R} f, \qquad \mathcal{R}:=\left(\mc{R}_{ij}\right)_{1\leq i, j \leq M},
\end{equation*}
actually provides us gives with the unique solution of
\begin{equation*}
  \mc{P} u =\mc{C} u +f.
\end{equation*}
Thus, the proof of (C1) $\Rightarrow$ (C2) is complete.
\par
Obviously, (C2) $\Rightarrow$ (C3) $\Rightarrow$ (C4). Next, we will show that (C3) $\Rightarrow$  (C5). According to (C3),
the operator
\begin{equation*}
 (\mc{P}-\mc{C})^{-1} :V\rightarrow V,
\end{equation*}
is compact and strongly order preserving. Thus, by Theorem \ref{thKR}(a,b),
\[
\mathrm{spr\,} (\mc {P}-\mc{C})^{-1} > 0
\]
provides us with the unique real eigenvalue of $(\mc {P}-\mc{C})^{-1}$ to a positive eigenfunction, $\Phi$.
Moreover, $\Phi\gg 0$. The principal eigenvalue whose existence is claimed in (C5) is given by
\begin{equation*}
  \s [\mc {P}-\mc{C}]:=\frac{1}{\mathrm{spr\,} (\mc {P}-\mc{C})^{-1} }.
\end{equation*}
Now, we will show  that  (C5) $\Rightarrow$ (C1). Let $\Phi\in \mathrm{int\,}\mc {P}_U $ be such that
\begin{equation*}
\mc{P} \Phi = \mc{C} \Phi + \s[\mc{P}-\mc{C}] \Phi.
\end{equation*}
Then, $ \mc{P} \Phi > \mc{C} \Phi $, as requested in (C1).
\par
Finally, we will show that (C4) $\Rightarrow$ (C1).
As we are assuming that \eqref{vi.6} satisfies de maximum principle, $u=0$ is the unique solution of
\begin{equation*}
\mc{P} u = \mc{C} u.
\end{equation*}
Therefore, by the  open mapping Theorem,
\begin{equation*}
  \mc{P}-\mc{C} : U\rightarrow V
\end{equation*}
is a topological  isomorphism, as it is  Fredholm of index zero.
\par
Fix $p \in \mathrm{int\,}\mc{P}_U$. As  \eqref{vi.6} satisfies the maximum principle, we have that
\begin{equation}
\label{6.25}
   \Psi := (\mc{P} - \mc{C})^{-1}p\in P_U\setminus\{0\}.
\end{equation}
Let $\l > 0$ be large enough so that
\begin{equation}
\label{6.26}
  \s[\mc{P}_k \!-\! c_{kk}\!+\! \l,\mf{B}_k,Q_T]\!=\! \s[\mc{P}_k \!-\! c_{kk},\mf{B}_k,Q_T]\!+\! \l >0, \;\;
   k \in \{1,...,M\}.
\end{equation}
It suffices to consider
\begin{equation*}
  \l =\max\left\{1,  \max_{1\leq k \leq M} \Big( - \s[\mc{P}_k - c_{kk},\mf{B}_k,Q_T]\Big) \right\}.
\end{equation*}
Let $(\psi_1,...,\psi_M)$ be the coordinates of $\Psi$. According to \eqref{6.25},
\begin{equation*}
(\mc{P}_k - c_{kk}+ \l) \psi_k = \l \psi_k+p_k+\sum_{j=1, j\neq k}^M c_{kj} \psi_j, \qquad 1\leq k \leq M.
\end{equation*}
Thus,
\begin{equation*}
 \psi_k = (\mc{P}_k - c_{kk}+ \l)^{-1}\left( \l \psi_k+p_k+\sum_{j=1, j\neq k}^M c_{kj} \psi_j\right), \qquad 1 \leq k \leq M.
\end{equation*}
Therefore, since $p_k\gg 0$, it follows from Theorem \ref{th12} that
\begin{equation*}
    \psi_k \in \mathrm{int\, } \mc{P}_{E_{\mf{B}_k}}, \qquad 1\leq  k \leq M.
\end{equation*}
As, due to \eqref{6.25},
\begin{equation*}
    \mc{P} \Psi=\mc{C} \Psi + p>\mc{C} \Psi,
\end{equation*}
$\Psi$ provides us with the supersolution we were looking for and the proof of the theorem is concluded.
\end{proof}

Subsequently, the principal eigenvalue  $\s [\mc {P}-\mc{C}]$ will be denoted by
\[
   \s [\mc {P}-\mc{C}]:=\s [\mc {P}-\mc{C},\mc{B},Q_T]
\]
to emphasize its dependence on the boundary operator $\mc{B}$.

\section{Fundamental properties of the principal eigenvalues}

This section applies theorem \ref{th6.1} to obtain some fundamental properties of the principal eigenvalues. Among them, their
monotonicity and continuity properties. The next result provides us with the monotonicity of the principal
eigenvalue with respect to the coupling matrix $\mc{C}$.

\begin{theorem}
\label{th7.1}
Let
\begin{equation*}
  \mc{C}_1=\left(c_{ij}^1\right)_{i,j=1,...,M},\qquad \mc{C}_2=\left(c_{ij}^2\right)_{i,j=1,...,M},
\end{equation*}
be two cooperative matrixes of order $M$ such that  $\mc{C}_1>\mc{C}_2$ in the sense that
\begin{equation*}
 c_{ij}^1\geq c_{ij}^2  \qquad \hbox{for all} \;\; i, j \in \left\{1,...,M\right\}
\end{equation*}
and $c_{ij}^1\neq c_{ij}^2$ for some   $i, j\in \left\{1,...,M\right\}$. Then,
\begin{equation}
\label{7.1}
 \s [\mc{P} - \mc{C}_1,\mc{B},Q_T]< \s [\mc{P} - \mc{C}_2,\mc{B},Q_T].
\end{equation}

\end{theorem}

\begin{proof}
Let $\Phi_1 \in \mathrm{int\,}\mc {P}_U$ be a positive eigenfunction associated to the principal eigenvalue
 $\s [\mc{P} - \mc{C}_1,\mc{B},Q_T]$. Then,
\begin{align*}
  &  \left(\mc{P} -\mc{C}_2-\s [\mc{P} - \mc{C}_1,\mc{B},Q_T]I_{\R^M}\right)\Phi_1 \\ &\hspace{1.4cm} = \left(\mc{P}
  -\mc{C}_1- \s [\mc{P} - \mc{C}_1,\mc{B},Q_T]I_{\R^M} \right)\Phi_1+   \left(\mc{C}_1-\mc{C}_2\right)\Phi_1,
\end{align*}
where $I_{\R^M}$ stands for the identity map of $\R^M$. Thus,
\begin{equation*}
  \left(\mc{P}-\mc{C}_2-\s [\mc{P} - \mc{C}_1,\mc{B},Q_T]I_{\R^M}\right)\Phi_1 = \left(\mc{C}_1-\mc{C}_2\right)\Phi_1>0,
\end{equation*}
since $\Phi_1\gg 0$. Hence, $\Phi_1$ is an strongly positive supersolution of
\begin{equation*}
  \left( \mc{P}-\mc{C}_2-\s [\mc{P} - \mc{C}_1,\mc{B},Q_T]I_{\R^M},\mc{B},Q_T\right).
\end{equation*}
Consequently, by Theorem \ref{th6.1}, we may infer that
\begin{equation*}
  \s [\mc{P}-\mc{C}_2-\s [\mc{P} - \mc{C}_1,\mc{B},Q_T]I_{\R^M},\mc{B},Q_T]>0.
\end{equation*}
Therefore,
\begin{align*}
  0 & <\s [\mc{P}-\mc{C}_2-\s [\mc{P} - \mc{C}_1,\mc{B},Q_T]I_{\R^M},\mf{B},Q_T]\\ & = \s [\mc{P} - \mc{C}_2,\mc{B},Q_T]- \s [\mc{P} - \mc{C}_1,\mc{B},Q_T],
\end{align*}
which ends the proof.
\end{proof}

The next two corollaries of Theorem \ref{7.1} establish the continuous dependence of the principal eigenvalue with respect to $\mc{C}$. Subsequently, for any given Banach space, $X$, and integer $M\geq 1$, $\mf{M}_M(X)$ will stand for the set of
matrices of order $M$ with entries in $X$.

\begin{corollary}
\label{co7.1}
Let $\mc{C}_n\in \mf{M}_M(F)$, $n\geq 1$, be a sequence of matrices of cooperative type such that
\[
  \lim_{n\rightarrow \infty}\mc{C}_n=\mc{C} \quad \hbox{in} \quad \mf{M}_M(L^\infty(Q_T)).
\]
Then,
\begin{equation*}
 \lim_{n\rightarrow \infty}\s [\mc{P} - \mc{C}_n,\mc{B},Q_T]=\s [\mc{P} - \mc{C},\mc{B},Q_T].
\end{equation*}
\end{corollary}

\begin{proof}
For every $\e >0$ there exists a natural number $n_0=n_0(\e)>1$ such that
\begin{equation*}
  \left \|\mc{C}_n-\mc{C} \right \|_{\mf{M}_M(L^\infty(Q_T))} \leq \e\quad \hbox{for all} \;\; n\geq n_0.
\end{equation*}
Thus,
\begin{equation*}
  \mc{C}-\e \leq \mc{C}_n \leq \mc{C}+\e \quad \hbox{for all} \;\; n\geq n_0\;\;\hbox{in}\;\; Q_T.
\end{equation*}
Hence, due to Theorem \ref{th7.1},
\begin{equation*}
\s [\mc{P} - \mc{C}-\e,\mc{B},Q_T] \leq \s [\mc{P} - \mc{C}_n,\mc{B},Q_T] \leq \s [\mc{P} - \mc{C}+\e,\mc{B},Q_T], \quad n\geq n_0.
\end{equation*}
So, we can infer that
\begin{equation*}
\s [\mc{P} - \mc{C},\mc{B},Q_T]-\e \leq \s [\mc{P} - \mc{C}_n,\mc{B},Q_T] \leq \s [\mc{P} - \mc{C},\mc{B},Q_T]+\e, \quad n\geq n_0,
\end{equation*}
and the proof is completed.
\end{proof}

\vspace{0.4cm}
\begin{corollary}
\label{co7.2}
Let $\l_n \in \R$, $n\geq 1$, be a sequence such that
\[
   \lim_{n\rightarrow \infty} \l_n =\l\in\R
\]
and $\mc{C}\in \mf{M}_M(F)$ a cooperative matrix. Then,
\[
\lim_{n\rightarrow \infty}\l_n \mc{C}=\l\mc{C} \quad \hbox{in} \quad \mf{M}_M(L^\infty(Q_T)).
\]
Therefore, by Corollary \ref{co7.1},
\begin{equation*}
\lim_{n\rightarrow \infty}\s [\mc{P} - \l_n\mc{C},\mc{B},Q_T]=\s [\mc{P} -\l \mc{C},\mc{B},Q_T].
\end{equation*}
Consequently, the mapping $\l \mapsto \s [\mc{P} -\l \mc{C},\mc{B},Q_T]$ is continuous.
\end{corollary}

Subsequently, we will adapt Propositions 3.1, 3.2  and 3.5 of \cite{CL-02} to our setting here. Essentially, they establish the monotonicity of the principal eigenvalue with respect to  $\b$ and with respect to the underlying domain $\O$, as well as the dominance of the principal eigenvalue under homogeneous Dirichlet boundary conditions.
To state their periodic-parabolic counterparts we need to introduce some of notation. For every $k\in\{1,...,M\}$, we will emphasize the dependence of the boundary operator $\mc{B}$ on the weight function $\b_k$ by setting
\[
 \mc{B}(\b_k):=\mc{B}
\]
if $\G_{1,k}\neq \emptyset$. Note that $\mf{B}_k=\mf{D}_k$ if $\G_{1,k} = \emptyset$. The next result shows the monotonicity of the principal eigenvalue with respect to each of these $\b_k$'s, separately.

\begin{theorem}
\label{th7.2} Let $\mc{C}\in \mf{M}_M(F)$ be a cooperative matrix and suppose $\G_{1,k}\neq \emptyset $ for some $k\in \{1,...,M\}$.
Let $\b_{k,i} \in \mathcal{C}^{1+\t}(\G_{1,k})$, $i=1, 2$, be satisfying $\b_{k,1}<\b_{k,2}$. Then,
\begin{equation*}
   \s [\mc{P} -\mc{C},\mc{B}(\b_{k,1}),Q_T]< \s [\mc{P} -\mc{C},\mc{B}(\b_{k,2}),Q_T].
\end{equation*}
\end{theorem}

\begin{proof} Let
\[
  \Phi_1=(\Phi_{11},....,\Phi_{1M}) \gg 0
\]
be a principal eigenfunction associated to $\s [\mc{P} -\mc{C},\mc{B}(\b_{k,1}),Q_T]$. Then,
\[
  (\mc{P}-\mc{C}-\s [\mc{P} -\mc{C},\mc{B}(\b_{k,1}),Q_T])\Phi_1=0 \qquad \hbox{in}\;\; Q_T.
\]
Moreover, since $\mf{B}_j(\b_{k,2}) = \mf{B}_j(\b_{k,1})$ if $j \neq k$, we have that
\[
  \mf{B}_j(\b_{k,2}) \Phi_{1j}=0 \qquad \hbox{on}\;\;\p\O \;\; \hbox{if}\;\; j \in\{1,...,M\}\setminus\{k\}.
\]
Lastly,
\[
  \Phi_{1k} =0 \qquad \hbox{on}\;\; \G_{0,k}
\]
and
\begin{align*}
  \mf{B}_k(\b_{k,2})\Phi_{1k} & = \frac{\p \Phi_{1k}}{\p \nu}+\b_{k,2}\Phi_{1k} \cr & =
  \frac{\p \Phi_{1k}}{\p \nu}+\b_{k,1}\Phi_{1k}+(\b_{k,2}-\b_{k,1})\Phi_{1k}\cr &= (\b_{k,2}-\b_{k,1})\Phi_{1k} >0
\end{align*}
because $\b_{k,2}-\b_{k,1}>0$ and $\Phi_{1k}(x,t)>0$ for all $x \in\O\cup \G_{1,k}$. Thus, $\Phi_1$ provides us
with a strict positive supersolution of
\[
  (\mc{P}-\mc{C}-\s [\mc{P} -\mc{C},\mc{B}(\b_{k,1}),Q_T],\mc{B}(\b_{k,2}),Q_T).
\]
Therefore, according to Theorem \ref{th6.1}, we find that
\begin{align*}
  0 & < \s[\mc{P}-\mc{C}-\s [\mc{P} -\mc{C},\mc{B}(\b_{k,1}),\O],\mc{B}(\b_{k,2}),Q_T] \cr
  & =  \s[\mc{P}-\mc{C},\mc{B}(\b_{k,2}),Q_T]-\s [\mc{P} -\mc{C},\mc{B}(\b_{k,1}),Q_T].
\end{align*}
This ends the proof.
\end{proof}

The next result establishes that the Dirichlet eigenvalue is the greatest one among all possible
principal eigenvalues.

\begin{theorem}
\label{th7.3}
 Let $\mc{C}\in \mf{M}_M(F)$ be a cooperative matrix and suppose $\G_{1,k}\neq \emptyset $ for some $k\in \{1,...,M\}$. Then,
\begin{equation*}
   \s [\mc{P} -\mc{C},\mc{B},Q_T]< \s [\mc{P} -\mc{C},\mc{D},Q_T].
\end{equation*}
\end{theorem}

\begin{proof}
Let $\Phi_{[\mc{P} -\mc{C},\mc{B},Q_T]}\gg 0$ and  $\Phi_{[\mc{P} -\mc{C},\mc{D},Q_T]}\gg 0$ be two principal eigenfunctions associated with $\s [\mc{P} -\mc{C},\mc{B},Q_T]$ and $\s [\mc{P} -\mc{C},\mc{D},Q_T]$ respectively, and let denote by
$\Phi_{[\mc{P} -\mc{C},\mc{B},Q_T],k}$ and  $\Phi_{[\mc{P} -\mc{C},\mc{D},Q_T],k}$ their  $k$-th components.
By Corollary \ref{co31},
\[
 \Phi_{[\mc{P} -\mc{C},\mc{B},Q_T],k}(x) >0 \quad \hbox{for all}\quad x\in \O\cup \G_{1,k},\;\; t\in[0,T].
\]
Hence,
\[
 \mf{D}_k \Phi_{[\mc{P} -\mc{C},\mc{B},Q_T],k}(x)= \Phi_{[\mc{P} -\mc{C},\mc{B},Q_T],k}(x)>0 \quad \hbox{for all}\quad  x\in
 \G_{1,k},\;\; t\in[0,T],
\]
and, consequently,
\[
   \mc{D}\Phi_{[\mc{P} -\mc{C},\mc{B},Q_T]}= \Phi_{[\mc{P} -\mc{C},\mc{B},Q_T]}>0 \qquad \hbox{in}\;\; \p \O\times [0,T].
\]
Thus, $\Phi_{[\mc{P} -\mc{C},\mc{B},Q_T]}$ provides us with a positive strict supersolution of
\begin{equation*}
  (\mc{P}-\mc{C}-\s [\mc{P} -\mc{C},\mc{B},Q_T],\mc{D},Q_T)
\end{equation*}
and, therefore, Theorem \ref{th6.1} yields
\begin{equation*}
 0<\s[\mc{P} -\mc{C}-\s [\mc{P} -\mc{C},\mc{B},Q_T],\mc{D},Q_T]=
 \s[\mc{P} -\mc{C},\mc{D},Q_T]-\s [\mc{P} -\mc{C},\mc{B},Q_T].
\end{equation*}
This concludes the proof.
\end{proof}

Suppose $\G_{1,k}\neq \emptyset$ for some $1\leq k \leq M$. For every proper subdomain $\O_0$ of class $\mathcal{C}^{2+\t}$ of $\O$ with
\begin{equation}
\label{7.2}
 \rm{dist\,}(\G_{1,k},\partial \O_0\cap \O)>0,
\end{equation}
we will denote by $\mf{B}_k[\b_k,\O_0]$, or simply by $\mf{B}_k[\O_0]$, the boundary operator
\begin{equation}
\label{7.3}
    \mf{B}_k[\O_0] \xi := \left\{ \begin{array}{ll}
    \xi \qquad & \hbox{on } \;\;\partial \O_0 \cap \O \\     \mf{B}_k  \xi \qquad &
   \hbox{on } \;\; \;\partial \O_0 \cap \partial\O  \end{array} \right.
\end{equation}
for each $\xi\in \mathcal{C}(\G_{0,k})\oplus \mathcal{C}^1(\O\cup\G_{1,k})$, where $\b_k\in \mathcal{C}^{1+\t}(\G_{1,k})$.
If $\bar \O_0 \subset \O$, then $\partial\O_0 \subset \O$ and hence,
\[
 \mf{B}_k[\O_0] \xi = \xi,
\]
for every $\xi\in \mathcal{C}(\G_{0,k})\oplus \mathcal{C}^1(\O\cup\G_{1,k})$, i.e, $\mf{B}_k[\O_0]$ becomes the Dirichlet boundary operator $\mf{D}_k$ in $\O_0$.
\par
When $\G_{1,k}=\emptyset$, i.e, $\mf{B}_k=\mf{D}_k$, then we define
\begin{equation}
\label{7.4}
 \mf{B}_k[\O_0] :=\mf{D}_k.
\end{equation}
Note that \eqref{7.3} becomes \eqref{7.4} if $\G_{1,k}=\emptyset$. Finally, we consider
\begin{equation}
\label{vii.5}
  \mc{B}[\O_0]:=(\mf{B}_1[\O_0],\ldots,\mf{B}_M[\O_0]).
\end{equation}
The following result establishes the monotonicity of the principal eigenvalue with respect to $\O$.
\begin{theorem}
\label{th7.4}
 Let $\mc{C}\in \mf{M}_M(F)$ be a cooperative matrix and $\O_0$ a proper subdomain of $\O$ of class $\mc{C}^{2+\t}$ satisfying \eqref{7.2} whenever $\G_{1,k}\neq \emptyset $. Then,
\begin{equation*}
     \s [\mc{P} -\mc{C},\mc{B},Q_T]< \s [\mc{P} -\mc{C},\mc{B}[\O_0],\O_0 \times [0,T]],
\end{equation*}
where $\mc{B}[\O_0]$ is the boundary operator defined by \eqref{vii.5}.
\end{theorem}

\begin{proof}
Let $\Phi:=(\Phi_1,\ldots,\Phi_M)\gg 0$ be a principal eigenfunction associated with $\s [\mc{P} -\mc{C},\mc{B},Q_T]$. By
definition,
\[
  (\mc{P}-\mc{C}-\s [\mc{P} -\mc{C},\mc{B},Q_T])\Phi=0  \quad \hbox{in} \;\; \O_0 \times [0,T],
\]
because $\O_0\subset\O$. Moreover,  for every $1\leq \ell \leq M$, we have that
\begin{equation*}
 \left\{ \begin{array}{ll} \Phi_\ell >0 & \quad \hbox{in} \;\; (\p \O_0\cap \O)\times [0,T],\\ \Phi_\ell =0 & \quad \hbox{in} \;\; (\p \O_0\cap \G_{0,\ell})\times [0,T],\\ \p_{\nu}\Phi_\ell + \b_\ell\Phi_\ell =0 & \quad \hbox{in} \;\; (\p \O_0\cap \G_{1,\ell})\times [0,T], \end{array}\right.
\end{equation*}
by construction. Note that $\p \O_0\cap \O \neq \emptyset$, because  $\O_0 \varsubsetneq\O$. Thus, $\Phi|_{\O_0}$ is a positive strict supersolution of
\[
(\mc{P}-\mc{C}-\s [\mc{P} -\mc{C},\mc{B},Q_T],\mc{B}[\O_0],\O_0\times [0,T]).
\]
Therefore, it follows from Theorem \ref{th6.1} that
\begin{align*}
  0 & < \s [\mc{P}-\mc{C}-\s [\mc{P} -\mc{C},\mc{B},Q_T],\mc{B}[\O_0],\O_0\times [0,T]] \cr
   & = \s [\mc{P} -\mc{C},\mc{B}[\O_0],\O_0 \times [0,T]]-\s [\mc{P} -\mc{C},\mc{B},Q_T].
\end{align*}
This ends the proof.
\end{proof}
Finally, as an immediate consequence of Theorems \ref{th7.2} and \ref{th7.4} the next result holds.
\begin{corollary}
\label{co7.3}
 Let $\mc{C}\in \mf{M}_M(F)$ be a cooperative matrix and suppose $\G_{1,k}\neq \emptyset $ for some $k\in \{1,...,M\}$. Then,
for every subdomain of class $\mathcal{C}^{2+\t}$ of $\O$, $\O_0$, satisfying  \eqref{7.2} if $\G_{1,k}\neq \emptyset$, and
any $\b_{k,i} \in \mathcal{C}^{1+\t}(\G_{1,k})$, $i=1, 2$, such that $\b_{k,1}<\b_{k,2}$,
\begin{equation}
\label{7.5}
   \s [\mc{P} -\mc{C},\mc{B}[\b_{k,1},\O],Q_T]< \s [\mc{P} -\mc{C},\mc{B}[\b_{k,2},\O_0],\O_0\times[0,T]].
\end{equation}
The same conclusion holds if $\b_{k,1}\leq\b_{k,2}$ and $\O_0\subsetneq\O$.
\end{corollary}

\section{The case $\mc{P}_1=\cdots=\mc{P}_M$ with $\mf{B}_1=\cdots = \mf{B}_M$}
Throughout this section we assume that
\[
 \mc{P}_1(x,t;D)=...=\mc{P}_M(x,t;D),\qquad \mf{B}_1=\cdots = \mf{B}_M.
\]
The following result characterizes whether or not the problem \eqref{vi.5} satisfies the strong maximum principle in the especial case
when $\mc{C}(x,t)$ is non-spatial.

\begin{theorem}
\label{th8.1}
Suppose  $C(x,t)\equiv C(t)$. Then, the problem  \eqref{vi.5} satisfies the strong maximum principle if and only if
\begin{equation}
\label{8.1}
 \l_1[\mc{P}_1,\mf{B}_1,Q_T] >\frac{1}{T} \log\mu_1[\mc{C}]
\end{equation}
where  $\mu_1[\mc{C}]$ stands for the principal eigenvalue of the monodromy matrix of the system
\begin{equation}
\label{8.2}
 \alpha'(t)=\mc{C}(t)\alpha(t),
\end{equation}
whose existence is guaranteed by Theorem 4.7 de M. A. Krasnoselskij \cite{Kr}.
\end{theorem}

\begin{proof}
Let $\s[\mc{P}-\mc{C};\mf{B},Q_T]$ be the principal eigenvalue of
\begin{equation*}
   \mc{P} \phi = \mc{C}(t) \phi+ \sigma \phi ,\quad \phi \in U.
\end{equation*}
We claim that
\begin{equation}
\label{8.3}
  \s[\mc{P}-\mc{C},\mf{B},Q_T] =\l_1[\mc{P}_1,\mf{B}_1,Q_T]-\frac{1}{T} \log\mu_1[\mc{C}].
\end{equation}
Thanks to Theorem  \ref{th6.1}, the proof of Theorem \ref{th8.1} is an immediate consequence of \eqref{8.3}.
\par
Let $\varphi \in \mathrm{int\,}\mc {P}_E$, $\mf{B}\v=0$,  be a principal eigenfunction associated  to $\l_1[\mc{P}_1;\mf{B}_1,Q_T]$. To establish \eqref{8.3}, we shall look for values of
\begin{equation*}
  \alpha(t)=(\alpha_1(t),...,\alpha_M(t)), \;\; \alpha_i(t)>0, \;\; \alpha_i(t+T)=\alpha_i(t), \;\; t\in \R, \;\; i=1,...,M,
\end{equation*}
so that
\begin{equation}
\label{8.4}
\Phi(x,t):=(\alpha_1(t)\varphi(x,t),...,\alpha_M(t)\varphi(x,t))
\end{equation}
satisfy
\begin{equation}
\label{8.5}
  \mc{P} \Phi = \mc{C}(t) \Phi + \s[\mc{P}-\mc{C},\mf{B},Q_T]\Phi.
\end{equation}
The function $\Phi(x,t)$ satisfies \eqref{8.5} if and only if
\begin{align*}
\varphi(x,t)\alpha'(t)   + \l_1[\mc{P}_1,\mf{B}_1,Q_T]\varphi(x,t)\alpha(t) & =\varphi(x,t)\mc{C}(t)\alpha(t)\\ & \quad +
 \s[\mc{P}-\mc{C},\mf{B},Q_T] \varphi(x,t)\alpha(t)
\end{align*}
for every $(x,t)\in Q_T$, which can be equivalently written down as
\begin{equation}
\label{8.6}
   \alpha'(t)=(\mc{C}(t)+\s[\mc{P}-\mc{C},\mf{B},Q_T] - \l_1[\mc{P}_1,\mf{B}_1,Q_T])\,\alpha(t),\qquad t \in [0,T].
\end{equation}
Subsequently, for every $x\in\R^N$, we denote by $\a(t;x)$ the unique solution of \eqref{8.2} such that $\a(0;x)=x$. Let $W(t)$ be the operator  defined by
\[
  W(t)x= \a(t;x)
\]
for all $t \in\R$  and $x \in\R^N$, i.e. $W(t)$ is the translation operator along the trajectories of \eqref{8.2}, or simply the Poincar\'e map of \eqref{8.2}. Obviously, $W(t)$ is a fundamental matrix solutions of the system \eqref{8.2} with $W(0)=I$, and the matrix
\begin{equation*}
  V(t):=e^{(\s[\mc{P}-\mc{C},\mc{B},Q_T] -\l_1[\mc{P}_1,\mf{B}_1,Q_T])t}W(t)
\end{equation*}
is the solution operator of \eqref{8.6}. Consequently,
\begin{equation}
\label{8.7}
  V(T):=e^{(\s[\mc{P}-\mc{C},\mc{B},Q_T] -\l_1[\mc{P}_1,\mf{B}_1,Q_T])T}W(T)
\end{equation}
is the monodromy operator of \eqref{8.6}.
\par
The system \eqref{8.6} has a positive $T$-periodic solution if and only if there exists some $x >0$ such that $V(T)x=x$. Thanks to \eqref{8.7}, this is equivalent to
\begin{equation*}
  W(T)x = e^{-(\s[\mc{P}-\mc{C},\mc{B},Q_T] -\l_1[\mc{P}_1,\mf{B}_1,Q_T])T} x.
\end{equation*}
Therefore, by the uniqueness of the principal eigenvalue for the monodromy operator, as established by
Theorem 4.6 of M. A. Krasnoselskij \cite{Kr}, it becomes apparent that
\begin{equation*}
\mu_1[\mc{C}]=e^{-(\s[\mc{P}-\mc{C},\mc{B},Q_T] -\l_1[\mc{P}_1,\mf{B}_1,Q_T])T}.
\end{equation*}
Therefore, \eqref{8.3} holds. This ends the proof.
\end{proof}

\vspace{0.4cm}

Estimating the  principal eigenvalue  of $W(T)$ is far from being, in general, an easy task. However, in the especial case when
\begin{equation}
\label{8.8}
  \int_0^s \mc{C} \; \int_0^t \mc{C}  = \int_0^t \mc{C}  \; \int_0^s \mc{C}  \quad \hbox{for every} \;\; t,s \in \R,
\end{equation}
it is well known that
\begin{equation}
\label{8.9}
  W(t)= \exp  \int_0^t \mc{C},\qquad t \in \R.
\end{equation}
Thus, the next result holds.

\begin{theorem}
\label{th8.2}
Suppose   $\mc{C}(x,t)\equiv\mc{C}(t)$ and \eqref{8.8} holds., and consider the  matrix
\begin{equation}
\label{8.10}
   \mc{M}:=\l_1[\mc{P}_1,\mf{B}_1,Q_T] I_{\R^M} -\frac{1}{T} \int_0^T\mc{C}(s)\, ds.
\end{equation}
Then, the following conditions are equivalent:
\begin{itemize}
\item[{\rm (i)}]  The problem \eqref{vi.5} satisfies the strong maximum principle.
\item[{\rm (ii)}] The following estimate holds
\begin{equation*}
  \l_1[\mc{P}_1,\mf{B}_1,Q_T]> \mathrm{spr\,}\left(\frac{1}{T} \int_0^T \mc{C}(s)\,ds \right),
\end{equation*}
where $\mathrm{spr} A$ stands for the spectral radius of $A$.
\item[{\rm (iii)}] The matrix $\mc{M}^{-1}$ define an strong positive endomorphism of $\R^M$.
\item[{\rm (iv)}] The principal minors of the matrix $\mc{M}$ have a positive determinant.
\item[{\rm (v)}] The $M$ first minors of $\mc{M}$ have a positive determinant.
\end{itemize}
\end{theorem}

\begin{proof}
Condition (ii) implies that the matrix defined in \eqref{8.10} is a non-singular M-matrix. Some well known properties of M-matrices are collected in F. R. Gantmacher \cite{Gantmacher}. From these properties
the equivalence between (ii), (iv) and (v) is easy. The implications (iii) $\Rightarrow$ (iv) and (iv) $\Rightarrow$ (iii)  were already established by  D. G. Figueiredo and E. Mittidieri  \cite{FM}  and  J. L\'{o}pez-G\'omez  and M. Molina-Meyer \cite{LM}, respectively. It remains to show that (i) $\Leftrightarrow$ (ii).
\par
As we are supposing  \eqref{8.8}, \eqref{8.9} holds. Thus, by the spectral mapping theorem (see, e.g., Theorem 2.4.2 of \cite{LMB}),
\begin{equation*}
\mu_1[\mc{C}]=\mathrm{spr\,} W(T)=\exp \left(\mathrm{spr} \int_0^T \mc{C}(s)\,ds \right).
\end{equation*}
Therefore, by \eqref{8.3}, we find that
\begin{equation*}
\s[\mc{P}-\mc{C},\mc{B},Q_T] =\l_1[\mc{P}_1,\mf{B}_1,Q_T] -\frac{1}{T}\; \mathrm{spr} \int_0^T \mc{C}(s)\,ds.
\end{equation*}
As  Theorem \ref{th6.1} ensures us that  \eqref{8.10} satisfies the strong maximum principle if and only if
\[
  \s[\mc{P}-\mc{C},\mc{B},Q_T] >0
\]
the proof is complete.
\end{proof}

We conclude this paper with the next consequence of Theorems \ref{th7.1} and \ref{th8.1}.

\begin{corollary}
\label{co8.1}
Suppose $\mc{C}(x,t)$ is a general cooperative matrix and set
\[
  \mc{C}_S(t):=\left( \max_{x\in\bar\O}c_{ij}(x,t) \right)_{1\leq i, j\leq M}\qquad t\in [0,T].
\]
Then, \eqref{vi.5} satisfies the maximum principle if
\begin{equation}
\label{8.11}
   \l_1[\mc{P}_1,\mf{B}_1,Q_T]> \frac{1}{T} \log\mu_1[\mc{C}_S].
\end{equation}
\end{corollary}
\begin{proof}
By Theorems \ref{th7.1} and \ref{th8.1},
\[
  \s[\mc{P}-\mc{C},\mc{B},Q_T]\geq \s[\mc{P}-\mc{C}_S,\mc{B},Q_T]=\l_1[\mc{P}_1,\mf{B}_1,Q_T]- \frac{1}{T} \log\mu_1[\mc{C}_S].
\]
Thus, \eqref{8.11} implies $\s[\mc{P}-\mc{C},\mc{B},Q_T]>0$. Theorem \ref{th6.1} ends the proof.
\end{proof}

\end{document}